\documentclass[12pt,a4paper,reqno]{amsart}
\usepackage[T1]{fontenc}
\usepackage{amsmath,amssymb,ifthen}
\usepackage{fullpage}
\usepackage{graphicx,psfrag,subfigure}
\usepackage[dvips]{color}
\def\A{\mathbb A}

\def\R{{\mathbb R}}
\def\N{{\mathbb N}}

\def\MM{{\mathcal M}}

\def\OO{{\mathcal O}}
\def\PP{{\mathcal P}}

\def\SS{{\mathcal S}}
\def\TT{{\mathcal T}}

\def\matrix#1{{\boldsymbol{#1}}}
\def\vector#1{{\boldsymbol{#1}}}
\def\operator#1{{\mathcal{#1}}}

\def\diam{{\rm diam}}

\def\norm#1#2{\|#1\|_{#2}}

\def\set#1#2{\big\{#1\,:\,#2\big\}}

\def\eps{\varepsilon}

\newcommand{\normLtwo}[3][]{#1\|#2#1\|_{L^2(#3)}}

\newcommand{\enorm}[2][]{#1|\hspace*{-.5mm}#1|\hspace*{-.5mm}#1|#2#1|\hspace*{-.5mm}#1|\hspace*{-.5mm}#1|}

\def\normL2#1#2{\|#1\|_{L^2(#2)}}
\def\normHme#1#2{\|#1\|_{H^{-1}(#2)}}
\newcommand{\dual}[3][]{#1\langle#2\,,\,#3#1\rangle}

\def\myd{{\rm d}\hspace{-0.6mm}{\rm l}}

\def\Verfuerth{Verf\"urth}

\newcounter{constantsnumber}
\def\namec#1#2{%
 \ifthenelse{\equal{#1}{rel}}{C_{\rm rel}}{%
  \ifthenelse{\equal{#1}{mesh}}{C_{\rm mesh}}{%
  \ifthenelse{\equal{#1}{sz}}{C_{\rm sz}}{%
  \ifthenelse{\equal{#1}{drel}}{C_{\rm dRel}}{%
  \ifthenelse{\equal{#1}{eff}}{C_{\rm eff}}{%
  \ifthenelse{\equal{#1}{continuous}}{C_{\rm cont}}{%
  \ifthenelse{\equal{#1}{opt}}{C_{\rm opt}}{%
  \ifthenelse{\equal{#1}{norm}}{C_{\rm norm}}{%
  \ifthenelse{\equal{#1}{reliable}}{C_{\rm rel}}{%
  \ifthenelse{\equal{#1}{efficient}}{C_{\rm eff}}{%
  \ifthenelse{\equal{#1}{dlr}}{C_{\rm dlr}}{%
  \ifthenelse{\equal{#1}{inv}}{C_{\rm stab}}{%
  \ifthenelse{\equal{#1}{reduction}}{C_{\rm red}}{%
   \ifthenelse{\equal{#1}{unibound}}{C_{\rm hot}}{%
    \ifthenelse{\equal{#1}{elliptic}}{C_{\rm ell}}{%
   \ifthenelse{\equal{#1}{garding}}{C_{\mbox{\rm\scriptsize g\r{a}rd}}}{%
  \ifthenelse{\equal{#1}{refined}}{C_{\rm ref}}{%
  \ifthenelse{\equal{#1}{estred}}{C_{\rm est}}{%
  \ifthenelse{\equal{#1}{optimal}}{C_{\rm opt}}{%
  \ifthenelse{\equal{#1}{qo}}{C_{\rm qo}}{%
  \ifthenelse{\equal{#1}{rconv}}{C_{\rm conv}}{%
  \ifthenelse{\equal{#1}{cea}}{C_{\mbox{\rm\scriptsize C\'ea}}}{%
  \ifthenelse{\equal{#1}{nllip}}{C_{\mbox{\rm\scriptsize lip}}}{%
  \ifthenelse{\equal{#1}{nlelliptic}}{C_{\rm mon}}{%
  \ifthenelse{\equal{#1}{nlbound}}{C_{\ell oc}}{%
  \ifthenelse{\equal{#1}{almost}}{C}{%
  \ifthenelse{\equal{#1}{help}}{C}{%
  \ifthenelse{\equal{#2}{newcounter}}{\refstepcounter{constantsnumber}\label{const#1}}{}C_{\ref{const#1}}}%
}}}}}}}}}}}}}}}}}}}}}}}}}}}
\def\setc#1{\namec{#1}{newcounter}}
\def\c#1{\namec{#1}{reference}}
\def\definec#1{\refstepcounter{constantsnumber}\label{const#1}}%

\newcounter{contractionnumber}
\def\nameq#1#2{%
  \ifthenelse{\equal{#1}{estred}}{q_{\rm est}}{%
  \ifthenelse{\equal{#1}{estconv}}{q_{\rm est}}{%
  \ifthenelse{\equal{#1}{rconv}}{q_{\rm conv}}{%
  \ifthenelse{\equal{#1}{doerfler}}{q_{\rm D}}{%
  \ifthenelse{\equal{#2}{newcounter}}{\refstepcounter{contractionnumber}\label{contraction#1}}{}q_{\ref{contraction#1}}}%
}}}}
\def\setq#1{\nameq{#1}{newcounter}}
\def\q#1{\nameq{#1}{reference}}
\def\namerr#1#2{%
  \ifthenelse{\equal{#1}{elliptic}}{\rho_{\rm ell}}{%
 \ifthenelse{\equal{#1}{garding}}{\rho_{\mbox{\scriptsize \rm g\r{a}rd}}}{%
  \ifthenelse{\equal{#1}{estconv}}{\rho_{\rm est}}{%
  \ifthenelse{\equal{#1}{cea}}{\rho_{\mbox{\scriptsize C\'ea}}}{%
  \ifthenelse{\equal{#1}{qo}}{\rho_{\mbox{\scriptsize qo}}}{%
  \ifthenelse{\equal{#2}{newcounter}}{\refstepcounter{contractionnumber}\label{contraction#1}}{}\rho_{\ref{contraction#1}}}%
}}}}}
\def\setrr#1{\namerr{#1}{newcounter}}
\def\rr#1{\namerr{#1}{reference}}

\def\osc{{\rm osc}}

\newtheorem{theorem}{Theorem}
\newtheorem{proposition}[theorem]{Proposition}
\newtheorem{lemma}[theorem]{Lemma}

\newtheorem{algorithm}[theorem]{Algorithm}

\newenvironment{remark}{\medskip\noindent\textbf{Remark.}\ \it}{\qed\smallskip}
\newenvironment{example}{\medskip\noindent\textbf{Example.}\ \it}{\qed\smallskip}

\def\T{\mathbb T}

\def\subsection#1{\bigskip

\refstepcounter{subsection}{\bf\thesubsection.~#1.~~}}

\def\subsubsection#1{\bigskip

\refstepcounter{subsubsection}{\bf\thesubsubsection.{~#1.}~~}}

\begin{document}
\title[Adaptive FEM with optimal convergence rates]%
{Adaptive FEM with optimal convergence rates\\for a certain class 
of non-symmetric and possibly non-linear problems}
\date{\today}

\author{M.~Feischl}
\author{T.~F\"uhrer}
\author{D.~Praetorius}
\address{Institute for Analysis and Scientific Computing,
      Vienna University of Technology,
      Wiedner Hauptstra\ss{}e 8-10,
      A-1040 Wien, Austria}

\email{Michael.Feischl@tuwien.ac.at\quad\rm(corresponding author)}
\email{\{Thomas.Fuehrer,\,Dirk.Praetorius\}@tuwien.ac.at}


\begin{abstract}
We analyze adaptive mesh-refining algorithms for conforming finite element
discretizations of certain non-linear second-order partial differential equations. We allow continuous polynomials of arbitrary, but fixed polynomial order. The
adaptivity is driven by the residual error estimator. We prove convergence
even with optimal algebraic convergence rates. In particular, our analysis covers general linear second-order elliptic operators. Unlike prior works for linear non-symmetric operators, our
analysis avoids the interior node property for the refinement, and the 
differential operator has to satisfy a G\r{a}rding inequality only. If the
differential operator is uniformly elliptic, no additional assumption on the
initial mesh is posed.
\end{abstract}
\maketitle

\section{Introduction}

Let $\Omega$ be a bounded polyhedral Lipschitz domain in $\R^d$, $d\ge2$. 
We consider a homogeneous Dirichlet boundary value problem for a certain non-linear
second-order elliptic partial differential equation (PDE)
\begin{subequations}
\label{intro:modelproblem}
\begin{align}\label{intro:L}
\operator{L} u(x):= -\textrm{div}\big(\matrix{A}(x,\nabla u)\big) + g(x,u,\nabla u) &= f(x)
\quad\text{in }\Omega,\\
\label{intro:boundary}
u&=0
\quad\quad\;\,\text{on }\Gamma:=\partial\Omega.
\end{align}
\end{subequations}
The differential operator $\operator{L} = \operator{A} + \operator{K}$ is 
split into a principal part 
$\operator{A}u = -\textrm{div}\big(\matrix{A}(\cdot,\nabla u)\big)$ and a compact perturbation 
$\operator{K}u = g(\cdot,u,\nabla u)$, see Subsection~\ref{section:nonlin} for the precise regularity assumptions.
This framework also includes the case of general linear second-order elliptic operators
\begin{align}\label{intro:linearL}
\operator{L} u:= -\textrm{div}(\matrix{A}\nabla u) + \vector{b}\cdot\nabla u + cu. 
\end{align}
We consider a common
adaptive mesh-refining algorithm which iterates the following loop
\begin{align}\label{intro:afem}
 \boxed{\texttt{ solve }}
 \quad\longrightarrow\quad
 \boxed{\texttt{ estimate }}
 \quad\longrightarrow\quad
 \boxed{\texttt{ mark }}
 \quad\longrightarrow\quad
 \boxed{\texttt{ refine }}
\end{align}
The module \texttt{solve} computes a piecewise polynomial finite element
approximation $U_\ell$ of $u$ with respect to a given mesh $\TT_\ell$.
For \texttt{estimate}, we use a residual error estimator, see 
e.g.~\cite{ao00,v96}. Next, the D\"orfler marking criterion~\cite{doerfler} 
is used to single out elements for refinement. Finally, \texttt{refine} 
leads to a locally refined and improved mesh $\TT_{\ell+1}$ by means of
the newest vertex bisection algorithm (NVB).

So far, available results on
convergence and quasi-optimality of adaptive finite element methods (AFEM) 
from the literature essentially dealt with the linear, symmetric, and elliptic case~\eqref{intro:linearL} with 
$\vector{b}=0$ and $c\ge0$, see 
e.g.~\cite{bdd,bn,ckns,doerfler,ks,stevenson07}
and the references therein. As far as the linear and non-symmetric 
case $\vector{b}\neq0$ is concerned, we are only aware of the works~\cite{cn,mn}
which, however, considered the 
special situation $\textrm{div}\,\vector{b}=0$ and $c\ge0$. Moreover, their 
analysis requires the interior node property for the refinement at least after 
a fixed number of steps, which has been introduced in~\cite{mns} to guarantee 
a discrete lower bound for the error. Finally, the proofs of convergence 
and quasi-optimality in~\cite{cn,mn} assume the initial mesh $\TT_0$ to be
sufficiently fine although the assumption $\textrm{div}\,\vector{b}=0$ already
ensures ellipticity of the associated bilinear form $b(\cdot,\cdot)$ in the 
weak formulation of~\eqref{intro:modelproblem}, i.e.\ the operator $\operator{L}$ in~\eqref{intro:linearL} 
is uniformly elliptic. 
All this is different to the present work, and the advances over the state 
of the art, see e.g.~\cite{ckns,cn,ks}, are fourfold:
\begin{itemize}
\item[(i)] In the linear case~\eqref{intro:linearL}, our assumptions on the data $\matrix{A} = \matrix{A}(x)$,
$\vector{b}=\vector{b}(x)$, and $c=c(x)$ only ensure that the bilinear form
$b(\cdot,\cdot)$ of the weak formulation of~\eqref{intro:modelproblem} is
continuous and satisfies a G\r{a}rding inequality on $H^1_0(\Omega)$.
\item[(ii)] As for the symmetric case~\cite{ckns}, we only rely on standard 
newest vertex bisection, and the interior node property is avoided.
\item[(iii)] If $b(\cdot,\cdot)$ is elliptic, we avoid any assumption on the
initial mesh $\TT_0$. If $b(\cdot,\cdot)$ satisfies a G\r{a}rding inequality, we
require the same assumption on the initial mesh as~\cite{cn,mn} to ensure
well-posedness of the finite element formulations.
\item[(iv)] To the best of the authors' knowledge and besides~\cite{bdk} for the particular $p$-Laplace problem, this work provides the first quasi-optimality result for a class of non-linear problems. 
\end{itemize}
From a technical point of view, our analytical argument works as follows and is illustrated for the linear operator $\operator{L}$ from~\eqref{intro:linearL} with induced bilinear form $b(\cdot,\cdot)$:
First, the estimator reduction
\begin{align}\label{intro:estconv}
 \eta_{\ell+1}^2 \le q\,\eta_\ell^2 + C\,\enorm{U_{\ell+1}-U_\ell}^2
\end{align}
together with a C\'ea-type quasi-optimality already implies convergence 
$U_\ell \to u$ as $\ell\to\infty$ (Proposition~\ref{prop:conv}),
see also~\cite{estconv} for this \emph{estimator reduction principle}. 
Here, $0<q<1$ and $C>0$ are generic constants, and 
$\enorm\cdot$ denotes the energy quasi-norm induced by $b(\cdot,\cdot)$.
Second, the novel contribution in our analysis is that this additional 
knowledge allows us to prove a quasi-Pythagoras theorem
\begin{align}\label{intro:orthogonality}
 \enorm{U_{\ell+1}-U_\ell}^2
 + \enorm{u-U_{\ell+1}}^2
 \leq \frac{1}{1-\eps}\,\enorm{u-U_\ell}^2
\end{align}
for all $\eps>0$ and $\ell\ge\ell_0(\eps)$ sufficiently large
(Proposition~\ref{prop:quasiqo})
which unlike~\cite{cn,mn} avoids any additional assumption on the mesh-size
of $\TT_\ell$. With estimator reduction~\eqref{intro:estconv} 
and quasi-orthogonality~\eqref{intro:orthogonality} 
at hand, we next observe $R$-linear convergence 
\begin{align}\label{intro:linear}
 \eta_{\ell+k} \le Cq^k\eta_\ell
 \quad\text{for all }\ell,k\in\N
\end{align}
of the error estimator (Theorem~\ref{thm:rconv}) with further generic constants
$C>0$ and $0<q<1$. 
Finally, the $R$-linear convergence~\eqref{intro:linear} suffices to follow 
the paths of~\cite{stevenson07,ckns}
to prove even quasi-optimal convergence rates in the sense of
\begin{align}\label{intro:optimal}
 (u,f) \in \A_s
 \quad\Longleftrightarrow\quad
 \eta_\ell \le C\,(\#\TT_\ell-\#\TT_0)^{-s}
 \quad\text{for all }\ell\in\N,
\end{align}
i.e.\ each theoretically possible convergence order $\OO(N^{-s})$ for the 
error estimator will asymptotically be achieved by AFEM. The approximation
class $\A_s$ involved in~\eqref{intro:optimal} is defined in Section~\ref{section:optimality}.
By means of reliability and efficiency of the error estimator $\eta_\ell$ used, 
this quasi-optimality result can equivalently be stated in terms of error plus
oscillations as is done in~\cite{ckns,cn,ks,stevenson07}. As has first been
observed in~\cite{dirichlet3d}, our approach and proof of~\eqref{intro:optimal}, however, fully 
avoids the use of lower bounds for the error, i.e.\ all constants are 
independent of the efficiency estimate.

For the nonlinear problem~\eqref{intro:modelproblem}, we observe that estimator reduction~\eqref{intro:estconv}, $R$-linear convergence~\eqref{intro:linear}, as well as quasi-optimality~\eqref{intro:optimal} do not hinge on linearity of $\operator{L}$. We thus bootstrap the arguments developed for the linear case to prove a quasi-Pythagoras theorem~\eqref{intro:orthogonality} for nonlinear $\operator{L}$ (Proposition~\ref{prop:nlquasiqo}), and may derive convergence of AFEM with quasi-optimal algebraic rates.

The remainder of this paper is organized as follows: 
For the sake of a clear presentation, we first consider the linear case~\eqref{intro:linearL} with elliptic bilinear form $b(\cdot,\cdot)$ corresponding to the weak formulation of~\eqref{intro:modelproblem}. This case already includes the main ideas of how
to cope with compact perturbations. 
In Section~\ref{section:modelproblem}, we explicitly state the assumptions on
the differential operator $\operator{L}$ from~\eqref{intro:linearL}, recall the 
continuous and discrete variational formulation of~\eqref{intro:modelproblem},
and give the necessary details on the four modules of~\eqref{intro:afem}.
Section~\ref{section:convergence} then provides the estimator 
reduction~\eqref{intro:estconv}, which follows as in~\cite{ckns}, and 
the quasi-Galerkin orthogonality~\eqref{intro:orthogonality} which relies 
on the convergence of AFEM and compactness arguments. 
The short Section~\ref{section:contraction} proves $R$-linear 
convergence~\eqref{intro:linear} of the error estimator by use 
of~\eqref{intro:estconv}--\eqref{intro:orthogonality}.
We stress that, so far, the analysis does neither hinge on the precise 
mesh-refinement used, nor on the adaptivity parameter chosen.
By use of intrinsic properties of NVB, we then prove quasi-optimal convergence
rates~\eqref{intro:optimal} in Section~\ref{section:optimality}.
A final Section~\ref{section:extensions} is concerned with extensions of
our analysis. Amongst other topics, we discuss other boundary conditions 
than~\eqref{intro:boundary} as well as changes of our analysis if the bilinear form $b(\cdot,\cdot)$
satisfies only a G\r{a}rding inequality. Subsection~\ref{section:nonlin} bootstraps the arguments of the previous sections and incorporates the non-linear case~\eqref{intro:L} into the analysis.

In all statements, the constants involved and their dependencies are explicitly stated. In proofs, however, we use the symbol $\lesssim$ to abbreviate $\leq$ up to a multiplicative constant. Moreover, $\simeq$ abbreviates that both estimates $\lesssim$ and $\gtrsim$ hold.

\section{Model Problem \& Adaptive Algorithm}
\label{section:modelproblem}
This section is devoted to state the model problem~\eqref{intro:modelproblem} with linear differential operator~\eqref{intro:linearL} in weak form and to collect all the ingredients needed to formulate the adaptive algorithm. The presented problem is not the most general case on which the developed theory can be applied, but it allows for a rather simple presentation and illustrates the main difficulties of the problem. We refer to Section~\ref{section:extensions} for possible extensions and generalizations.
\subsection{Variational formulation}
For a given right-hand side $f\in L^2(\Omega)$, we consider  the elliptic boundary value problem~\eqref{intro:modelproblem} with linear operator $\operator{L}$ from~\eqref{intro:linearL}.
For the weak formulation, the error estimator, and to prove optimal convergence rates, we require some regularity assumptions on the coefficients. We assume that
$\matrix{A}=\matrix{A}(x) \in\R^{d\times d}$ with $\matrix{A}\in \big(W_1^\infty(\Omega)\big)^{d\times d}$ is a symmetric matrix, $\vector{b}=\vector{b}(x) \in \R^d$ with $\vector{b}\in \big(L^\infty(\Omega)\big)^d$ is a vector, and $c=c(x)\in\R$ with $c\in L^\infty(\Omega)$ is a scalar. Here, $W_1^\infty(\Omega):=\set{a \in L^\infty(\Omega)}{\nabla a \in \big(L^\infty(\Omega)\big)^d\text{ in the weak sense}}$ coincides with the space of Lipschitz continuous functions.
This allows to write down the weak formulation of~\eqref{intro:modelproblem}: Find $u\in H^1_0(\Omega):=\set{v\in H^1(\Omega)}{v|_{\Gamma}=0\text{ in the sense of traces}}$ such that 
\begin{align}\label{eq:weakform}
b(u,v):= \int_\Omega \matrix{A}\nabla u\cdot\nabla v + \vector{b}\cdot\nabla u\,v + cu v\,dx = \int_\Omega fv\,dx\quad\text{for all }v\in H^1_0(\Omega).
\end{align}
According to Sobolev's embedding theorem, there holds $H^1_0(\Omega)\subset L^{2d/(d-2)}(\Omega)$. The bilinear form $b(\cdot,\cdot)$ is therefore well-defined and bounded
with
\begin{align}\label{eq:continuous}
|b(u,v)|\leq \c{continuous}\normLtwo{\nabla u}{\Omega}\normLtwo{\nabla v}{\Omega}\quad\text{for all }u,v\in H^1_0(\Omega),
\end{align}
where the constant $\setc{continuous}:=C_{\Omega}\big(\norm{\matrix{A}}{L^\infty(\Omega)} + \norm{\vector{b}}{L^{d/(d+2)}(\Omega)} + \norm{c}{L^{d/2}(\Omega)}\big)$ depends only on the coefficients of $\operator{L}$ as well as the Poincar\'e constant $C_\Omega>0$ of $\Omega$.
Additionally, we assume that the coefficients ensure that $b(\cdot,\cdot)$ is elliptic, i.e.
\begin{align}\label{eq:elliptic}
 b(u,u)\geq \c{elliptic} \normLtwo{\nabla u}{\Omega}^2\quad\text{for all }u\in H^1_0(\Omega)
\end{align}
for some constant $\setc{elliptic}>0$ which may also depend on $C_\Omega>0$, see Section~\ref{section:extensions} if $b(\cdot,\cdot)$ satisfies only a G\r{a}rding inequality.

Now, the Lax-Milgram lemma guarantees unique solvability of~\eqref{eq:continuous} for all $f\in L^2(\Omega)$ and proves continuous dependence $\normLtwo{\nabla u}{\Omega}\lesssim \norm{f}{ H^{-1}(\Omega)}\leq \normLtwo{f}{\Omega}$. Here, $H^{-1}(\Omega):= H^1_0(\Omega)^\star$ denotes the dual space of $H^1_0(\Omega)$, and duality is understood with respect to the extended $L^2$-scalar product, i.e. 
\begin{align*}
\normHme{f}{\Omega} := \sup_{v\in H^1_0(\Omega)\setminus\{0\}} \frac{\int_\Omega fv\,dx}{\normLtwo{\nabla v}{\Omega}}.
\end{align*}
Moreover, the bilinear form $b(\cdot,\cdot)$ defines a \textit{quasi}-norm $\enorm{\cdot}:=b(\cdot,\cdot)^{1/2}$, i.e.\ $\enorm{\cdot}$ is definite and homogeneous, but satisfies the triangle inequality only up to some multiplicative constant. Due to ellipticity and continuity of $b(\cdot,\cdot)$, it holds
\begin{align}\label{eq:normequiv}
\c{norm}^{-1}\normLtwo{\nabla v}{\Omega}\leq\enorm{v}\leq \c{norm} \normLtwo{\nabla v}{\Omega}\quad\text{for all } v\in H^1_0(\Omega)
\end{align}
for a constant $\setc{norm}=\max\{\c{continuous}^{1/2},\c{elliptic}^{-1/2}\}>0$.

\subsection{Discrete formulation}
For any regular triangulation $\TT_\ell$ of $\Omega$ (see Section~\ref{sec:mesh} below) and $p\geq 1$, we consider the piecewise polynomials
\begin{align*}
 \PP^p(\TT_\ell):=\set{V_\ell\in L^2(\Omega)}{\text{for all }T\in\TT_\ell,\,V_\ell|_T\text{ is a polynomial of degree at most }p}
\end{align*}
as well as the conforming ansatz and test-space
\begin{align*}
\SS^p_0(\TT_\ell):=\PP^p(\TT_\ell)\cap H^1_0(\Omega)\subset \mathcal{C}(\overline{\Omega}).
\end{align*}
Now, the discrete formulation of~\eqref{eq:continuous} reads: Find $U_\ell\in\SS^p_0(\TT_\ell)$ such that
\begin{align}\label{eq:discrete}
b(U_\ell,V_\ell)=\int_\Omega f\, V_\ell\, dx \quad\text{for all }V_\ell\in\SS^p_0(\TT_\ell).
\end{align}
As in the continuous case~\eqref{eq:continuous}, existence and uniqueness of $U_\ell$ follows from the Lax-Milgram lemma. Moreover, there holds the C\'ea lemma
\begin{align}\label{eq:cea}
 \normLtwo{\nabla(u-U_\ell)}{\Omega}\leq\frac{\c{continuous}}{\c{elliptic}}\,\min_{V_\ell\in\SS^p_0(\TT_\ell)} \normLtwo{\nabla(u-V_\ell)}{\Omega}.
\end{align}

\subsection{Error estimator}
We use the standard weighted-residual error estimator with the local contributions
\begin{align*}
\eta_\ell(T)^2:= |T|^{2/d}\normLtwo{\operator{L}|_TU_\ell-f}{T}^2 + |T|^{1/d}\normLtwo{[\matrix{A}\nabla U_\ell\cdot n]}{\partial T\cap\Omega}^2\quad\text{for all }T\in\TT_\ell,\,\ell\in\N.
\end{align*}
Here, $|T|$ is the $d$-dimensional volume of $T\in\TT_\ell$, and $[\matrix{A}\nabla U_\ell\cdot n]|_E:= \big(\matrix{A}\nabla U_\ell|_{T_1}\big)\cdot n_{T_1} + \big(\matrix{A}\nabla U_\ell|_{T_2}\big)\cdot n_{T_2}$ denotes the conormal jump over the facet $E:=T_1\cap T_2 $ for all $T_1,T_2\in\TT_\ell$, where $n_{T_1},\,n_{T_2}$ denote the outward pointing normal units on the respective element boundaries. Note that due to the regularity assumptions on the coefficients, there holds $\operator{L}|_TU_\ell \in L^2(T)$ for all $T\in\TT_\ell$. The error estimator $\eta_\ell$ is defined as the $\ell_2$-sum of the elementwise contributions
\begin{align*}
\eta_\ell^2 := \sum_{T\in\TT_\ell}\eta_\ell(T)^2.
\end{align*}
As shown in e.g.~\cite{ao00,v96}, the error estimator is reliable, i.e. for all regular triangulations $\TT_\ell$ and corresponding solutions $U_\ell$ of~\eqref{eq:discrete}, it holds
\begin{align}\label{eq:reliable}
\normLtwo{\nabla(u-U_\ell)}{\Omega}\leq \c{reliable}\eta_\ell
\end{align}
for a constant $\setc{reliable}>0$.
Moreover, $\eta_\ell$ is also efficient, i.e.
\begin{subequations}\label{eq:efficient}
\begin{align}\label{eq:efficienta}
 \c{efficient}^{-1}\eta_\ell \leq \normLtwo{\nabla(u-U_\ell)}{\Omega} + \osc_\ell(U_\ell)
\end{align}
for a constant $\setc{efficient}>0$ and oscillation terms
\begin{align}\label{eq:efficientb}
\osc_\ell(U_\ell)^2:=\sum_{T\in\TT_\ell} |T|^{2/d}\normLtwo{(1-\Pi_\ell^{p-1})(\operator{L}|_TU_\ell-f)}{T}^2,
\end{align}
\end{subequations}%
where $\Pi_\ell^{p-1}:\,L^2(\Omega)\to\PP^{p-1}(\TT_\ell)$ denotes the $L^2$-orthogonal projection.
 The constants $\c{reliable},\c{efficient}>0$ depend only on $\gamma$-shape regularity of $\TT_\ell$ (see Section~\ref{sec:mesh} below), the polynomial degree $p\geq 1$, and on $\Omega$. We stress that unlike~\cite{ckns,cn,ks}, efficiency~\eqref{eq:efficient} is not used  throughout our analysis. 
\subsection{Adaptive algorithm}
Now, we are in the position to formulate the adaptive algorithm~\eqref{intro:afem} in detail.
\begin{algorithm}\label{algorithm}
\textsc{Input:} Initial triangulation $\TT_0$ and adaptivity parameter $0<\theta\leq 1$.\\
\textbf{Loop: }For $\ell=0,1,2,\ldots$ do ${\rm (i)}-{\rm(iv)}$
\begin{itemize}
\item[\rm(i)] Compute discrete solution $U_\ell$  of~\eqref{eq:discrete}.
\item[\rm(ii)] Compute refinement indicators $\eta_\ell(T)$ for all $T\in\TT_\ell$.
\item[\rm(iii)] Determine set $\MM_\ell\subseteq\TT_\ell$ of minimal cardinality such that
\begin{align}\label{eq:doerfler}
 \theta\,\eta_\ell(T)^2 \le \sum_{T\in\MM_\ell}\eta_\ell(T)^2.
\end{align}
\item[\rm(iv)] Refine (at least) the marked elements $T\in\MM_\ell$ to obtain the triangulation $\TT_{\ell+1}$.
\end{itemize}
\textsc{Output:} Approximate solutions $U_\ell$ and error estimators
$\eta_\ell$ for all $\ell\in\N$.
\end{algorithm}
\subsection{Mesh refinement}\label{sec:mesh}
Given an initial mesh $\TT_0$ which is regular in the sense of Ciarlet, we construct the subsequent meshes $\TT_\ell$ by local refinement with the newest vertex bisection for simplicial meshes in $\R^d$, $d\geq 2$, see e.g.~\cite[Chapter 4]{v96} resp.~\cite{stevenson08}.
Consequently, the set of meshes which can be obtained reads
\begin{align}\label{eq:triangulations}
 \T := \set{\TT_\ell}{\TT_\ell\text{ is a refinement of }\TT_0}.
\end{align}
The finite subset of meshes with at most $N\in\N$ elements more than the initial mesh is defined as
\begin{align*}
 \T_N := \set{\TT_\ell\in\T}{\#\TT_\ell-\#\TT_0\le N}.
\end{align*}
The meshes $\TT_\ell\in\T$ are regular in the sense of Ciarlet and $\gamma$-shape regular in the sense of
\begin{align}\label{refinement:shaperegular}
 \gamma^{-1}\, |T|^{1/d}&\leq \diam(T)\leq \gamma\,|T|^{1/d}
\end{align}
for some $\gamma\geq 1$ which depends only on $\TT_0$.
A refined element $T\in\TT_\ell$ is split into at least two sons, i.e.\ we have
\begin{align}\label{refinement:sons}
\#(\TT_\star\setminus\TT_\ell) \leq \#\TT_\star-\#\TT_\ell
\end{align}
for all refinements $\TT_\star\in\T$ of $\TT_\ell\in\T$. 
 As a key property for the optimality proof, the crucial closure estimate, for the meshes generated by Algorithm~\ref{algorithm}, is satisfied
\begin{align}\label{refinement:closure}
 \#\TT_\ell - \#\TT_0
 \le\c{mesh}\,\sum_{j=0}^{\ell-1}\#\MM_{j}\quad\text{for all }\ell\in\N
\end{align}
with some constant $\setc{mesh}>0$ which depends only on $\TT_0$. For $d\geq3$, $\TT_0$ has to satisfy a certain condition on the reference edges, cf.~\cite{bdd,stevenson08}, while this assumption can be dropped for $d=2$, see the recent work~\cite{kpp}.
Finally, for two meshes $\TT_\ell,\TT_\star\in\T$ there is a coarsest common refinement
$\TT_\ell\oplus\TT_\star\in\T$ which satisfies
\begin{align}\label{refinement:overlay}
 \#(\TT_\ell\oplus\TT_\star)\le \#\TT + \#\TT' - \#\TT_0,
\end{align}
see~\cite{ckns,stevenson07}. We stress that newest-vertex bisection is a binary refinement rule, and the coarsest common refinement $\TT_\ell\oplus\TT_\star$ is just the overlay of both meshes.

\section{Convergence \& Quasi-Orthogonality}
\label{section:convergence}
The aim of this section is to prove convergence, without relying on symmetry properties of $\operator{L}$, which can be done by use of the concept of estimator reduction~\cite{estconv}.
To that end, we define the subspace $\SS^p_0(\TT_\infty)$ of $H^1_0(\Omega)$ which is \textit{theoretically} affected by Algorithm~\ref{algorithm} as
\begin{align}\label{eq:xinfty}
\SS^p_0(\TT_\infty):= \overline{\bigcup_{\ell\in\N} \SS^p_0(\TT_\ell)},
\end{align} 
where the closure is taken with respect to the $H^1$-norm. With convergence $U_\ell\to u$ and hence $u\in\SS^p_0(\TT_\infty)$ at hand, we are then able to prove a novel quasi-Galerkin orthogonality estimate~\eqref{eq:quasiqo}, which is sufficient to prove linear convergence~\eqref{eq:rconv} as well as optimal convergence rates~\eqref{eq:optimality}.
\subsection{Convergence}
The following result is proved in~\cite{ckns} for symmetric $\operator{L}$ and shows that the error estimator $\eta_\ell$ is contractive up to a certain perturbation.
\begin{lemma}\label{lem:estred}
There exist constants $0<\setq{estred}<1$ and $\setc{estred}>0$, such that there holds
\begin{align}\label{eq:estred}
\eta_{\ell+1}^2 \leq \q{estred} \eta_\ell^2 + \c{estred}\normLtwo{\nabla(U_{\ell+1}-U_\ell)}{\Omega}^2\quad\text{for all }\ell\in\N.
\end{align}
The constants $\q{estred}$ and $\c{estred}$ depend only on $\theta$, $\gamma$-shape regularity of $\TT_{\ell+1}$, the polynomial degree $p\in\N$, and on $\Omega$.
\end{lemma}
\begin{proof}
The proof follows verbatim the proof of~\cite[Corollary~3.4]{ckns}. Therefore, we give a rough sketch only.
The application of Young's inequality $2ab\leq a^2+b^2$ proves for $\delta>0$
\begin{align*}
\begin{split}
 \eta_{\ell+1}^2&\leq (1+\delta)\sum_{T^\prime\in\TT_{\ell+1}} \Big(|T^\prime|^{2/d}\normLtwo{\operator{L}|_{T^\prime}U_\ell -f}{T^\prime}^2+|T^\prime|^{1/d}\normLtwo{[\matrix{A}\nabla U_\ell\cdot n]}{\partial T^\prime\cap \Omega}^2\Big)\\
&\qquad+(1+\delta^{-1})\sum_{T^\prime\in\TT_{\ell+1}} \Big(|T^\prime|^{2/d}\normLtwo{\operator{L}|_{T^\prime}(U_{\ell+1}-U_\ell)}{T^\prime}^2\\
&\qquad\qquad+|T^\prime|^{1/d}\normLtwo{[\matrix{A}\nabla (U_{\ell+1}-U_\ell)\cdot n]}{\partial T^\prime\cap \Omega}^2\Big).
\end{split}
\end{align*}
By use of the regularity assumption on the coefficients and standard inverse estimates as well as the Poincar\'e inequality, we obtain
\begin{align}\label{eq:stable}
\begin{split}
 \eta_{\ell+1}^2&\leq(1+\delta)\sum_{T^\prime\in\TT_{\ell+1}} \Big(|T^\prime|^{2/d}\normLtwo{\operator{L}|_{T^\prime}U_\ell -f}{T^\prime}^2+|T^\prime|^{1/d}\normLtwo{[\matrix{A}\nabla U_\ell\cdot n]}{\partial T^\prime\cap \Omega}^2\Big)\\
&\leq(1+\delta)\sum_{T^\prime\in\TT_{\ell+1}} \Big(|T^\prime|^{2/d}\normLtwo{\operator{L}|_{T^\prime}U_\ell -f}{T^\prime}^2+|T^\prime|^{1/d}\normLtwo{[\matrix{A}\nabla U_\ell\cdot n]}{\partial T^\prime\cap \Omega}^2\Big)\\
&\qquad\qquad\qquad\qquad+(1+\delta^{-1})\c{inv}\normLtwo{\nabla(U_{\ell+1}-U_\ell)}{\Omega}^2.
\end{split}
\end{align}
The constant $\setc{inv}>0$ depends only on the $\gamma$-shape regularity of $\TT_{\ell+1}$, the norms 
$\norm{\matrix{A}}{W_1^\infty(\Omega)}^2, \norm{\vector{b}}{L^\infty(\Omega)}^2, \norm{c}{L^\infty(\Omega)}^2$, and on the polynomial degree $p\in\N$.
Next, the sum is split into two sums over $T^\prime \in \TT_\ell\cap\TT_{\ell+1}$ and $T^\prime \in\TT_{\ell+1}\setminus\TT_\ell$. We use the reduction of the element size $|T^\prime|\leq |T|/2$ for $T^\prime\subset T$ being a son of a refined element $T\in\TT_\ell\setminus\TT_{\ell+1}$. Since $\MM_\ell\subseteq \TT_\ell\setminus\TT_{\ell+1}$, one ends up with
\begin{align*}
 \eta_{\ell+1}^2&\leq (1+\delta)\Big(2^{-1/d}\hspace{-3mm}\sum_{T\in\TT_\ell\setminus\TT_{\ell+1}} \eta_\ell(T)^2+\hspace{-4mm}\sum_{T\in\TT_\ell\cap\TT_{\ell+1}}\eta_\ell(T)^2\Big)
+(1+\delta^{-1})\c{inv}\normLtwo{\nabla(U_{\ell+1}-U_\ell)}{\Omega}^2\\
&\leq(1+\delta)\Big(2^{-1/d}\sum_{T\in\MM_\ell} \eta_\ell(T)^2 +\sum_{T\in\TT_\ell\setminus\MM_\ell}\eta_\ell(T)^2\Big)
+(1+\delta^{-1})\c{inv}\normLtwo{\nabla(U_{\ell+1}-U_\ell)}{\Omega}^2\\
&\leq(1+\delta)\Big((2^{-1/d}-1)\sum_{T\in\MM_\ell} \eta_\ell(T)^2
+\eta_\ell^2\Big)+(1+\delta^{-1})\c{inv}\normLtwo{\nabla(U_{\ell+1}-U_\ell)}{\Omega}^2.
\end{align*}
Finally, D\"orfler marking~\eqref{eq:doerfler} proves~\eqref{eq:estred} with
\begin{align*}
 \q{estred}=\big(1-\theta(1-2^{-1/d}\big)(1+\delta)\in (0,1)\quad\text{and}\quad\c{estred}=(1+\delta^{-1})\c{inv}
\end{align*}
for $\delta>0$ sufficiently small.
\end{proof}
Adaptive algorithms of the type of Algorithm~\ref{algorithm} with nested ansatz spaces $\SS^p_0(\TT_\ell)\subseteq \SS^p_0(\TT_{\ell+1})$ have in common that there holds
\textit{a~priori} convergence. This has already been observed in the early work~\cite{bv} and has later also been used in~\cite{msv} to prove a general plain convergence result for AFEM.
\begin{lemma}\label{lem:apriori}
The sequence of Galerkin approximations  $U_\ell$ of Algorithm~\ref{algorithm} is convergent in $H^1_0(\Omega)$, i.e.\ there exists $u_\infty\in\SS^p_0(\TT_\infty)$ with
\begin{align}\label{eq:apriori}
U_\ell \to u_\infty \quad\text{as }\ell\to\infty.
\end{align}
\end{lemma}
\begin{proof}
The space $\SS^p_0(\TT_\infty)$ is a closed subspace of $H^1_0(\Omega)$ and therefore the Lax-Milgram lemma guarantees existence and uniqueness of a solution $u_\infty\in\SS^p_0(\TT_\infty)$ of~\eqref{eq:discrete} with test space $\SS^p_0(\TT_\infty)$ instead of $\SS^p_0(\TT_\ell)$. The Galerkin approximations $U_\ell$ are also Galerkin approximations of $u_\infty$, since $\SS^p_0(\TT_\ell)\subseteq \SS^p_0(\TT_\infty)$ for all $\ell\in\N$. Therefore, the C\'ea lemma shows
\begin{align*}
\normLtwo{\nabla(u_\infty-U_\ell)}{\Omega}\lesssim \min_{V_\ell\in\SS^p_0(\TT_\ell)} \normLtwo{\nabla (u_\infty-V_\ell)}{\Omega}\to 0
\end{align*}
as $\ell \to \infty$. 
\end{proof}
The combination of estimator reduction~\eqref{eq:estred} and a~priori convergence~\eqref{eq:apriori} yields convergence of Algorithm~\ref{algorithm}.

\begin{proposition}\label{prop:conv}
Algorithm~\ref{algorithm} is convergent in $H^1_0(\Omega)$, i.e.
\begin{align}\label{eq:conv}
U_\ell \to u \in H^1_0(\Omega)\quad\text{as }\ell\to\infty.
\end{align}
In particular, this implies $u=u_\infty\in\SS^p_0(\TT_\infty)$.
\end{proposition}
\begin{proof}
According to Lemma~\ref{lem:apriori}, the estimator reduction~\eqref{eq:estred} of Lemma~\ref{lem:estred} takes the form
\begin{align*}
 \eta_{\ell+1}^2\leq \q{estred}\eta_\ell^2 + \alpha_\ell
\end{align*}
with $\alpha_\ell\geq 0$ and $\lim_{\ell\to\infty}\alpha_\ell=0$. From this, elementary calculus proves $\lim_{\ell\to\infty} \eta_\ell=0$, see e.g.~\cite{estconv}. Finally, reliability~\eqref{eq:reliable} of $\eta_\ell$ concludes the proof.
\end{proof}
\subsection{Quasi-Galerkin orthogonality}\label{section:quasiqo}
The standard proof of the Pythagoras theorem $\enorm{u-U_{\ell+1}}^2+\enorm{U_{\ell+1}-U_\ell}^2=\enorm{u-U_\ell}^2$ relies on Galerkin orthogonality and symmetry of $b(\cdot,\cdot)$. The following lemmata provide a workaround for our case of a non-symmetric bilinear form $b(\cdot,\cdot)$. We stress that the quasi-orthogonality proof makes explicit use of the fact that we already have convergence $U_\ell\to u$ in $H^1_0(\Omega)$ and $u\in \SS^p_0(\TT_\infty)$.
\begin{lemma}\label{lem:compact}
 The operators $\operator{A}, \operator{K}:\, H^1_0(\Omega) \to H^{-1}(\Omega)$ are bounded. Moreover, $\operator{A}$ is symmetric and $\operator{K}$ is compact.
 \end{lemma}
\begin{proof}
 The symmetry of $\operator{A}$ is obvious, and both operators $\operator{A}$ and $\operator{K}$ are also bounded, i.e.
 \begin{align*}
 \normHme{\operator{A}v}{\Omega}&\leq \norm{\matrix{A}}{L^\infty(\Omega)}  \normLtwo{\nabla v}{\Omega},\\
  \normHme{\operator{K}v}{\Omega}&\leq \normLtwo{\operator{K}v}{\Omega}\leq  (\norm{\vector{b}}{L^{d/(d+2)}(\Omega)}+ \norm{c}{L^{d/2}(\Omega)}) \normLtwo{\nabla v}{\Omega},
 \end{align*}
for all $v\in H^1_0(\Omega)$.
It remains to prove that $\operator{K}$ is compact. 
The Rellich compactness theorem shows that the embedding $\iota:\,H^1_0(\Omega)\hookrightarrow L^2(\Omega)$ is a compact operator. Therefore, according to Schauder's theorem, see e.g.~\cite[Theorem~4.19]{rudin}, the adjoint operator $\iota^\star:\,L^2(\Omega)\to H^{-1}(\Omega)$ is also compact. Obviously, $\iota^\star:\,L^2(\Omega)\to H^{-1}(\Omega)$ coincides with the natural embedding, and we may write
\begin{align*}
\operator{K} = \iota^\star \circ \operator{K} :  H^1_0(\Omega)\to L^2(\Omega)\to H^{-1}(\Omega).
\end{align*}
Therefore, $\operator{K}$ is the composition of a bounded operator and a compact operator and hence compact. This concludes the proof.
\end{proof}

\begin{lemma}\label{lem:weakconv}
The sequences $(e_\ell)_{\ell\in\N}$ and $(E_\ell)_{\ell\in\N}$ defined by
\begin{align*}
e_\ell:=\begin{cases}\frac{u-U_\ell}{\normLtwo{\nabla(u-U_\ell)}{\Omega}},& \text{ for }u\neq U_\ell,\\
	0, &\text{ else,}\end{cases}\quad\text{and}\quad 
	E_\ell:=\begin{cases}\frac{U_{\ell+1}-U_\ell}{\normLtwo{\nabla(u-U_\ell)}{\Omega}},&\text{ for }U_{\ell+1}\neq U_\ell,\\
		0, &\text{ else,}\end{cases}
\end{align*}
converge to zero, weakly in $ H^1_0(\Omega)$.
\end{lemma}
\begin{proof}
We prove weak convergence of $e_\ell$ to zero. The weak convergence of $E_\ell$ follows with the same arguments.
Let $(e_{\ell_j})$ be a subsequence of $(e_\ell)$. Due to boundedness $\normLtwo{\nabla e_{\ell_j}}{\Omega}\leq 1$ for all $j\in\N$, we may extract a weakly convergent subsequence $(e_{\ell_{j_k}})$ of $(e_{\ell_j})$ with
\begin{align*}
e_{\ell_{j_k}} \rightharpoonup w \in H^1_0(\Omega).
\end{align*}
First, note that $u,U_\ell\in\SS_0^p(\TT_\infty)$ implies $e_\ell\in \SS_0^p(\TT_\infty)$ and hence $w\in\SS_0^p(\TT_\infty)$. Second, for all $\ell_{j_k}\geq \ell$ with $e_{\ell_{j_k}}\neq 0$ and all $V_\ell\in\SS^p_0(\TT_\ell)$, it holds
\begin{align*}
b(e_{\ell_{j_k}},V_\ell) = \normLtwo{\nabla(u-U_{\ell_{j_k}})}{\Omega}^{-1} b(u-U_{\ell_{j_k}}, V_\ell) = 0.
\end{align*}
For any $\ell\in\N$, $V_\ell\in\SS^p_0(\TT_\ell)$, and $\eps>0$, there exists $k_0\in\N$ such that for all $k\geq k_0$, it holds
\begin{align*}
|b(w,V_\ell) |=| \dual{w}{\operator{L}^\star V_\ell} | \leq \eps + |\dual{e_{\ell_{j_k}}}{\operator{L}^\star V_\ell} |= \eps + |b(e_{\ell_{j_k}},V_\ell)| = \eps,
\end{align*} 
since $k_0$ is chosen large enough such that ${\ell_{j_k}}\geq \ell$. Therefore
\begin{align*}
b(w,V_\ell)=0 \quad\text{for all }V_\ell\in\SS^p_0(\TT_\ell) \text{ and } \ell\in\N.
\end{align*}
Due to definiteness of $b(\cdot,\cdot)$ and $w\in\SS^p_0(\TT_\infty):=\overline{\bigcup_{\ell\in\N}\SS^p_0(\TT_\ell)}$, this implies $w=0$. Altogether, we have now shown that each subsequence of $e_\ell$ has a subsequence which converges weakly to zero. This immediately implies weak convergence $e_\ell \rightharpoonup 0$ as $\ell\to\infty$.
\end{proof}
The previous lemma shows that although $(E_\ell)_{\ell\in\N}$ is no orthonormal sequence, it shares the property of weak convergence to zero with orthonormal systems. Note that our proof already used convergence $U_\ell\to u$ as $\ell\to\infty$ in the sense that we required $u-U_\ell\in \SS^p_0(\TT_\infty)$. This suffices to prove the following quasi-Pythagoras theorem.

\begin{proposition}\label{prop:quasiqo}
For any $0<\eps<1$, there exists $\ell_0\in\N$ such that
\begin{align}\label{eq:quasiqo}
\enorm{U_{\ell+1}-U_\ell}^2\leq \frac{1}{1-\eps}\,\enorm{u-U_\ell}^2-\enorm{u-U_{\ell+1}}^2
\end{align}
for all $\ell\geq \ell_0$.
\end{proposition}
\begin{proof}
Lemma~\ref{lem:weakconv} shows that $e_\ell,E_\ell\rightharpoonup 0 $ as $\ell\to \infty$. Due to Lemma~\ref{lem:compact}, $\operator{K}$ is compact. Therefore, we have strong convergence  $\operator{K}e_\ell, \operator{K} E_\ell \to 0$ in $H^{-1}(\Omega)$ as $\ell\to\infty$ . This shows
\begin{align*}
\dual{\operator{K}(u-U_{\ell+1})}{U_{\ell+1}-U_\ell} &= \dual{\operator{K}e_{\ell+1}}{U_{\ell+1}-U_\ell} \normLtwo{\nabla(u-U_{\ell+1})}{\Omega}\\
&\leq \normHme{\operator{K}e_{\ell+1}}{\Omega} 
\normLtwo{\nabla(u-U_{\ell+1})}{\Omega}\normLtwo{\nabla(U_{\ell+1}-U_\ell)}{\Omega}
\end{align*}
as well as
\begin{align*}
\dual{\operator{K}(U_{\ell+1}-U_\ell)}{u-U_{\ell+1}} &= \dual{\operator{K}E_\ell}{u-U_{\ell+1}} \normLtwo{\nabla(U_{\ell+1}-U_\ell)}{\Omega}\\
&\leq \normHme{\operator{K}E_\ell}{\Omega} 
\normLtwo{\nabla(u-U_{\ell+1})}{\Omega}\normLtwo{\nabla(U_{\ell+1}-U_\ell)}{\Omega}.
\end{align*}
For any $\delta>0$, this may be employed to obtain some $\ell_0\in\N$ such that for all $\ell\geq \ell_0$, it holds
\begin{align*}
 |\dual{\operator{K}(U_{\ell+1}-U_\ell)}{u-U_{\ell+1}}|
&+ |\dual{\operator{K}(u-U_{\ell+1})}{U_{\ell+1}-U_\ell}|\\
&\leq \delta \normLtwo{\nabla(u-U_{\ell+1})}{\Omega}\normLtwo{\nabla(U_{\ell+1}-U_\ell)}{\Omega}.
\end{align*}
Together with Galerkin orthogonality 
\begin{align}\label{eq:galorth}
0=b(u-U_{\ell+1},V_{\ell+1})=\dual{\operator{L}(u-U_{\ell+1})}{V_{\ell+1}}\quad\text{for all }V_{\ell+1}\in\SS_0^p(\TT_{\ell+1}),
\end{align}
we estimate
\begin{align}\label{eq:key}
\begin{split}
|\dual{\operator{L}(U_{\ell+1}-U_\ell)}{u-U_{\ell+1}}|&= |\dual{\operator{A}(u-U_{\ell+1})}{U_{\ell+1}-U_\ell} + \dual{\operator{K}(U_{\ell+1}-U_\ell)}{u-U_{\ell+1}}|\\
&\leq  |\dual{\operator{L}(u-U_{\ell+1})}{U_{\ell+1}-U_\ell}| + |\dual{\operator{K}(U_{\ell+1}-U_\ell)}{u-U_{\ell+1}}|\\
&\qquad\qquad +  |\dual{\operator{K}(u-U_{\ell+1})}{U_{\ell+1}-U_\ell}|\\
&\leq \delta\normLtwo{\nabla(u-U_{\ell+1})}{\Omega}\normLtwo{\nabla(U_{\ell+1}-U_\ell)}{\Omega}.
\end{split}
\end{align}
The definition of $\enorm{\cdot}$ and Galerkin orthogonality~\eqref{eq:galorth} yield
\begin{align*}
\enorm{u-U_{\ell+1}}^2 + \enorm{U_{\ell+1}-U_\ell}^2 + \dual{\operator{L}(U_{\ell+1}-U_\ell)}{u-U_{\ell+1}}=\enorm{u-U_\ell}^2,
\end{align*}
whence
\begin{align*}
\enorm{U_{\ell+1}-U_\ell}^2\leq \enorm{u-U_\ell}^2-\enorm{u-U_{\ell+1}}^2 + \delta\c{norm}^2\enorm{u-U_{\ell+1}}\enorm{U_{\ell+1}-U_\ell}.
\end{align*}
The application of Young's inequality $2ab\leq a^2 +b^2 $ and the choice $\eps=\delta\c{norm}^2/2$ conclude the proof.
\end{proof}

\section{Contraction}\label{section:contraction}
The quasi-Pythagoras theorem~\eqref{eq:quasiqo} from Proposition~\ref{prop:quasiqo} allows to prove $R$-linear convergence of the error estimator $\eta_\ell$.
Compared with the analysis of the symmetric case~\cite{ckns}, this is a weaker result. However, $R$-linear convergence is still sufficient to prove quasi-optimal convergence rates in Section~\ref{section:optimality}.
\begin{theorem}\label{thm:rconv}
 There exist constants $0<\setq{rconv}<1$ and $\setc{rconv}>0$ such that for all $\ell,k\in\N$, there holds
\begin{align}\label{eq:rconv}
 \eta_{\ell+k}^2 \leq \c{rconv}\q{rconv}^k\, \eta_\ell^2.
\end{align}
The constants $\q{rconv}$ and $\c{rconv}$ depend only on $\q{estred}$, $\c{estred}$, $\c{norm}$, and $\c{reliable}$.
\end{theorem}
\begin{proof}
We employ the estimator reduction~\eqref{eq:estred} and reliability~\eqref{eq:reliable} to obtain for $N\geq \ell+1$ and $\alpha<1-\q{estred}$
\begin{align*}
 \sum_{k=\ell+1}^N \eta_k^2 &\leq \sum_{k=\ell+1}^N \big(\q{estred}\eta_{k-1}^2 +\c{estred}\normLtwo{\nabla(U_k-U_{k-1})}{\Omega}^2\big)\\
&\leq  \sum_{k=\ell+1}^N \Big((\q{estred}+\alpha)\eta_{k-1}^2 +\c{estred}\big(\normLtwo{\nabla(U_k-U_{k-1})}{\Omega}^2-\alpha\c{reliable}^{-2}\c{estred}^{-1}\normLtwo{\nabla(u-U_{k-1})}{\Omega}^2\big)\Big),
\end{align*}
Rearranging the terms in the above estimate, we end up with
\begin{align*}
 (1-\q{estred}-\alpha) \sum_{k=\ell+1}^N \eta_k^2 &
\leq (1+\q{estred}+\alpha)\eta_{\ell}^2 + \c{estred}\c{norm}^2\sum_{k=\ell+1}^N \big(\enorm{U_k-U_{k-1}}^2- \delta\enorm{u-U_{k-1}}^2\big).
\end{align*}
where $\delta= \alpha\c{reliable}^{-2}\c{estred}^{-1}\c{norm}^{-4}$.
Next, we aim at proving that the sum on the right-hand side is bounded above by $\eta_\ell^2$ for all $N\in\N$.
To that end, we employ Lemma~\ref{prop:quasiqo} with $\eps>0$ such that $1/(1-\eps) \leq 1+\delta$.
This gives a number $\ell_0\in \N$ such that for all $N>\ell\geq \ell_0$, we may estimate 
\begin{align}\label{eq:rconvhelp2}
 \sum_{k=\ell+1}^N \big(\enorm{U_{k}-U_{k-1}}^2 -\delta\enorm{u-U_{k-1}}^2\big)
&\leq\sum_{k=\ell+1}^N \big((\frac{1}{1-\eps}-\delta)\enorm{u-U_{k-1}}^2-\enorm{u-U_{k}}^2\big)\nonumber\\
&\leq\sum_{k=\ell+1}^N \big(\enorm{u-U_{k-1}}^2-\enorm{u-U_{k}}^2\big)\\
&\leq \enorm{u-U_\ell}^2\leq \c{norm}^2\c{reliable}^2\eta_\ell^2.\nonumber
\end{align}
For all $\ell<\ell_0$, we first observe that $\enorm{u-U_\ell}=0$ implies $\enorm{U_k-U_{k-1}}=0$ for all $k\geq \ell+1$, since $U_k=u=U_{k-1}$. Therefore, we obtain with the convention $\infty\cdot 0=0$
\begin{align*}
 C_{\rm sup}:=\sup_{\ell\in\{1,\ldots,\ell_0\}} \Big(\enorm{u-U_\ell}^{-2}\sum_{k=\ell+1}^{\ell_0}\enorm{U_{k}-U_{k-1}}^2\Big)<\infty.
\end{align*}
In combination with~\eqref{eq:rconvhelp2}, we thus see
\begin{align*}
 \sum_{k=\ell+1}^N& \big(\enorm{U_k-U_{k-1}}^2 -\delta\enorm{u-U_{k-1}}^2\big)\leq (1+C_{\rm sup})\c{norm}^2\c{reliable}^2\eta_\ell^2\quad\text{for all }\ell\in\N,\,N> \ell.
\end{align*}
Plugging everything together, we have so far shown
\begin{align}\label{eq:rconvfinal}
 \sum_{k=\ell+1}^\infty \eta_k^2 \leq\c{help}\eta_\ell^2\quad\text{for all }\ell\in\N,
\end{align}
for some constant $\setc{help}>0$ which depends only on $\q{estred}$, $\c{estred}$, $\c{norm}$, and $\c{reliable}$.
Therefore, we get
\begin{align*}
(1+\c{help}^{-1}) \sum_{k=\ell+1}^\infty \eta_k^2 \leq  \sum_{k=\ell+1}^\infty \eta_k^2  +\eta_\ell^2 =  \sum_{k=\ell}^\infty \eta_k^2,
\end{align*}
and hence by induction
\begin{align*}
 \eta_{\ell+j}^2\leq \sum_{k=\ell+j}^\infty \eta_k^2\leq (1+\c{help}^{-1})^{-j}\sum_{k=\ell}^\infty \eta_k^2\leq (1+\c{help})(1+\c{help}^{-j})^{-k}\eta_\ell^2
\quad\text{for all }\ell,k\in\N.
\end{align*}
This concludes the proof with $\q{rconv}=1/(1+\c{help}^{-1})$ and $\c{rconv}=(1+\c{help})$.
\end{proof}

\begin{remark}
Note that the $R$-linear convergence of Theorem~\ref{thm:rconv} holds for arbitrary adaptivity parameters $0<\theta<1$. Moreover, the result is independent of NVB in the sense that the proof only requires that $|T^\prime|\leq q|T|$ for some $0<q<1$ and all sons $T^\prime\subset T$ of refined elements $T\in\TT_{\ell}\setminus\TT_{\ell+1}$. This property holds for each feasible mesh-refinement strategy and for NVB with $q=2^{-1/d}$. Finally, the minimal cardinality of the set $\MM_\ell$ of marked elements has not been used, yet. Instead, Theorem~\ref{thm:rconv} holds as long as the set $\MM_\ell\subseteq \TT_\ell$ satisfies the D\"orfler marking~\eqref{eq:doerfler} and, in particular, for $\MM_\ell=\TT_\ell$.
\end{remark}

\begin{remark}
 Note that the proof of Theorem~\ref{thm:rconv} does neither use linearity  nor uniform ellipticity of $\operator{L}$. Instead, we only require reliability~\eqref{eq:reliable}, estimator reduction~\eqref{eq:estred}, quasi-Galerkin orthogonality~\eqref{eq:quasiqo} as well as equivalence~\eqref{eq:normequiv} of the norm $\normLtwo{\nabla(\,\cdot\,)}{\Omega}$ and the energy quasi-norm $\enorm{\cdot}$ on $H^1_0(\Omega)$. With these ingredients, our analysis is thus also capable to cover certain nonlinear problems as discussed in Section~\ref{section:nonlin}.
\end{remark}

\section{Optimal Convergence Rates}\label{section:optimality}
With Theorem~\ref{thm:rconv} at hand, we are in the position to prove quasi-optimal convergence rates for the sequence of Galerkin solutions obtained from Algorithm~\ref{algorithm}. First, however, we have to clarify what is the best possible convergence rate that can be aimed at. To that end, we follow e.g.~\cite{ckns} and define the approximation class $\A_s$ by
\begin{subequations}\label{eq:approxclasstotalerror}
\begin{align}\label{eq:approxclasstotalerrora}
 (u,f)\in\A_s\quad\overset{\rm def}{\Longleftrightarrow}\quad \norm{(u,f)}{\A_s}:=\sup_{N\in\N}N^s\sigma(N;u,f)<\infty
\end{align}
for all $s>0$, where
\begin{align}\label{eq:approxclasstotalerrorb}
 \sigma(N;u,f):=\inf_{\TT_\star\in\T_N}\inf_{V_\star\in\SS^p_0(\TT_\star)}\big(\normLtwo{\nabla(u-V_\star)}{\Omega}^2+\osc_\star(V_\star)^2\big)^{1/2}
\end{align}
\end{subequations}
and $\osc_\star$ is the oscillation term from~\eqref{eq:efficient} corresponding to the mesh $\TT_\star$. We refer to~\cite{bddp,gm} for a characterization of approximation classes in terms of Besov regularity. However, in this work, we follow~\cite{dirichlet3d} and use an equivalent definition of $\A_s$, which involves the error estimator $\eta_\ell$ only. This equivalence is part of the next lemma which is also implicitly contained in~\cite[Lemma~5.2]{ckns}.

\begin{lemma}\label{lem:totalerror}
There exists a constant $\setc{totalerror}>0$ such that for all $\TT_\star\in\T$ there holds
\begin{align}\label{eq:totalerror}
\c{totalerror}^{-1}\eta_\star^2\leq\inf_{V_\star\in\SS^p_0(\TT_\star)}\big(\normLtwo{\nabla(u-V_\star)}{\Omega}^2+\osc_\star(V_\star)^2\big)\leq\c{totalerror} \eta_\star^2.
\end{align} 
Hence, $\A_s$ from~\eqref{eq:approxclasstotalerror} can equivalently be characterized as 
\begin{align}\label{eq:approxclass}
(u,f)\in\A_s\quad\Longleftrightarrow\quad \sup_{N\in\N}\inf_{\TT_\star\in\T_N}\,N^s\eta_\star<\infty
\end{align}
for all $s>0$. The constant $\c{totalerror}$ depends only on $\c{continuous},\c{elliptic}$, the $\gamma$-shape regularity of $\TT_\star$ and the polynomial degree $p\in\N$.
\end{lemma}
\begin{proof}
First, we prove~\eqref{eq:totalerror}. To that end, we observe $ \normLtwo{\nabla(u-U_\star)}{\Omega}^2+\osc_\star(U_\star)^2\simeq \eta_\star^2$, which follows from reliability~\eqref{eq:reliable}, efficiency~\eqref{eq:efficient} as well as $\osc_\star(U_\star)\leq \eta_\star$. Moreover, the lower bound
 \begin{align}\label{eq:totalerrorhelp}
  \inf_{V_\star\in\SS^p_0(\TT_\star)}\big(\normLtwo{\nabla(u-V_\star)}{\Omega}^2+\osc_\star(V_\star)^2\big)\leq \normLtwo{\nabla(u-U_\star)}{\Omega}^2+\osc_\star(U_\star)^2
 \end{align}
holds since $U_\star\in\SS^p_0(\TT_\star)$. To prove the converse estimate in~\eqref{eq:totalerrorhelp}, we argue as in Lemma~\ref{lem:estred} and use a standard inverse estimate as well as the Poincar\'e inequality, to see
\begin{align*}
 \osc_\star(U_\star)^2&=\sum_{T\in\TT_\star}|T|^{2/d}\normLtwo{(1-\Pi_\star^{p-1})(\operator{L}|_TU_\star-f)}{T}^2\\
 &\lesssim\sum_{T\in\TT_\star}|T|^{2/d}\normLtwo{(1-\Pi_\star^{p-1})\operator{L}|_T(U_\star-V_\star)}{T}^2 + \osc_\star(V_\star)^2\\
 &\lesssim\big(\norm{\matrix{A}}{W_1^\infty(\Omega)}^2+\norm{\vector{b}}{L^\infty(\Omega)}^2+\norm{c}{L^\infty(\Omega)}^2\big)\sum_{T\in\TT_\star}|T|^{2/d}\norm{U_\star-V_\star}{H^2(T)}^2 + \osc_\star(V_\star)^2\\
 &\lesssim \normLtwo{\nabla(U_\star-V_\star)}{\Omega}^2 + \osc_\star(V_\star)^2.
\end{align*}
Finally, by use of the C\'ea lemma, we end up with
\begin{align*}
 \osc_\star(U_\star)^2&\lesssim \normLtwo{\nabla(u-U_\star)}{\Omega}^2 +\normLtwo{\nabla(u-V_\star)}{\Omega}^2 + \osc_\star(V_\star)^2\\
&\lesssim \normLtwo{\nabla(u-V_\star)}{\Omega}^2 + \osc_\star(V_\star)^2.
 \end{align*}
The combination of the last three estimates proves~\eqref{eq:totalerror}. The characterization~\eqref{eq:approxclass} follows with~\eqref{eq:totalerror} and the definition of $\sigma(N;u,f)$ in~\eqref{eq:approxclasstotalerrorb}.
\end{proof}

In our opinion, this characterization allows for a clearer presentation of the proof of the following quasi-optimality theorem and, in particular, we shall see that unlike the analysis of~\cite{ckns,cn,ks,stevenson07}, the upper bound for optimal adaptivity parameters $0<\theta<1$ does not depend on the efficiency constant $\c{efficient}$. The following result is the main theorem of this section.
\begin{theorem}\label{thm:optimal}
Define $\theta_\star:= (1+\c{inv}\c{drel})^{-1}$ with the constants $\c{drel}>0$ from Lemma~\ref{lem:drel} and $\c{inv}>0$ from the proof of Lemma~\ref{lem:estred}. 
Then, for all adaptivity parameters $0<\theta<\theta_\star$ and all $s>0$, there exists a constant $\setc{optimal}>0$ such that
\begin{align}\label{eq:optimality}
 (u,f)\in\A_s\quad\Longleftrightarrow\quad \eta_\ell \leq \c{optimal}\norm{(u,f)}{\A_s}(\#\TT_\ell-\#\TT_0)^{-s}\quad\text{for all }\ell\in\N.
\end{align}
The constant $\c{optimal}$ depends only on $\theta$, $s$, $\q{rconv}$, $\c{rconv}$, $\c{efficient}$, and $\c{mesh}$, and the proof relies on the properties~\eqref{refinement:shaperegular}--\eqref{refinement:overlay} of NVB.
\end{theorem}

For the proof of the quasi-optimality theorem, we need a refined reliability property of the error estimator $\eta_\ell$.
\begin{lemma}[discrete reliability]\label{lem:drel}
There exists a constant $\setc{drel}>0$ such that for all refinements $\TT_\star\in\T$ of a triangulation $\TT_\ell\in\T$, it holds
\begin{align}\label{eq:drel}
 \normLtwo{\nabla(U_\star-U_\ell)}{\Omega}^2\leq \c{drel}\sum_{T\in\TT_\ell\setminus\TT_\star}\eta_\ell(T)^2.
\end{align}
The constant $\c{drel}$ depends only on the $\gamma$-shape regularity of $\TT_0$, the polynomial degree $p\in\N$, and on $\Omega$.
\end{lemma}
\begin{proof}
 The statement is proven for $\vector{b}=0$ and $c\geq 0$ in~\cite[Lemma~3.6]{ckns}. The proof for the present case follows verbatim.
\end{proof}
So far, we have observed that D\"orfler marking~\eqref{eq:doerfler} implies contraction of $\eta_\ell$ (Proposition~\ref{thm:rconv}). Now, we prove, in some sense, the converse. We follow the concept of proof of~\cite{dirichlet3d} and stress that unlike e.g.~\cite{ckns,cn,ks,stevenson07} our proof does not use efficiency~\eqref{eq:efficient} of $\eta_\ell$.
\begin{lemma}[Optimality of D\"orfler marking]\label{lem:doerfler}
Let $0<\theta < \theta_\star:=(1+\c{inv}\c{drel})^{-1}$. Then, there exists $0<\setq{doerfler}<1$ such that for all refinements $\TT_\star\in\T$ of a triangulation $\TT_\ell\in\T$ the following statement is true
\begin{align}
\eta_\star^2 \leq \q{doerfler}\eta_\ell^2 \quad\implies\quad \theta\eta_\ell^2  \leq \sum_{T\in\TT_\ell\setminus\TT_\star} \eta_\ell(T)^2.
\end{align}
\end{lemma}
\begin{proof}
Analogously to~\eqref{eq:stable}, we estimate for $\delta>0$
\begin{align}\label{eq:mon}
\begin{split}
 \eta_\ell^2 &= \sum_{T\in\TT_\ell\setminus\TT_\star}\eta_\ell(T)^2+\sum_{T\in\TT_\ell\cap\TT_\star} \eta_\ell(T)^2\\
&\leq \sum_{T\in\TT_\ell\setminus\TT_\star}\eta_\ell(T)^2+(1+\delta^{-1})\sum_{T\in\TT_\ell\cap\TT_\star} \eta_\star(T)^2 + (1+\delta)\c{inv}\normLtwo{\nabla(U_\star-U_\ell)}{\Omega}^2\\
&\leq \sum_{T\in\TT_\ell\setminus\TT_\star}\eta_\ell(T)^2+(1+\delta^{-1})\q{doerfler}\eta_\ell^2+(1+\delta)\c{inv}\normLtwo{\nabla(U_\star-U_\ell)}{\Omega}^2.
\end{split}
\end{align}
Rearranging the terms and employing the discrete reliability~\eqref{eq:reliable}, we end up with
\begin{align*}
 \frac{1-(1+\delta^{-1})\q{doerfler}}{1+(1+\delta)\c{inv}\c{drel}}\,\eta_\ell^2\leq\sum_{T\in\TT_\ell\setminus\TT_\star}\eta_\ell(T)^2.
\end{align*}
According to $\theta<(1+\c{inv}\c{drel})^{-1}$, we may finally choose $\delta>0$ and $0<\q{doerfler}<1$ sufficiently small to ensure
\begin{align*}
 \theta \leq \frac{1-(1+\delta)\q{doerfler}}{1+(1+\delta^{-1})\c{inv}\c{drel}}< \frac{1}{1+\c{inv}\c{drel}}.
\end{align*}
This concludes the proof.
\end{proof}
Now, we are in the position to prove Theorem~\ref{thm:optimal}. We stress that the concept of proof goes back to~\cite{stevenson07} and has been adopted by~\cite{ckns} and all succeeding works. We put emphasis on the fact that, first, efficiency~\eqref{eq:efficient} of $\eta_\ell$ is not needed and that, second, $R$-linear convergence~\eqref{eq:rconv} instead of plain contraction in each step of the adaptive loop is sufficient.
\begin{proof}[Proof of Theorem~\ref{thm:optimal}]
Let $\lambda>0$ denote a free parameter, which is fixed later on.
 The definition of the approximation class $\A_s$ allows for given $\eps^2:=\lambda \eta_\ell^2>0$ to choose a mesh $\TT_\eps\in\T$ such that
\begin{align*}
 \eta_\eps \leq \eps\quad\text{and}\quad \#\TT_\eps-\#\TT_0 \lesssim \norm{(u,f)}{\A_s}^{1/s} \eps^{-1/s}.
\end{align*}
Now, consider the overlay $\TT_\star:=\TT_\eps\oplus\TT_\ell$ and argue similarly to~\eqref{eq:stable} to see
\begin{align*}
 \eta_\star^2 \lesssim \eta_\eps^2 + \normLtwo{\nabla(U_\star-U_\eps)}{\Omega}^2 \lesssim \eta_\eps^2 \leq \lambda \eta_\ell^2,
\end{align*}
where we used the definition of $\eps>0$. We choose $\lambda>0$ sufficiently small such that Lemma~\ref{lem:doerfler} is applicable and conclude that
$\TT_\ell\setminus\TT_\star$ satisfies the D\"orfler marking~\eqref{eq:doerfler}. By definition of step~(iii) of Algorithm~\ref{algorithm}, the set $\MM_\ell$ of marked elements is a set of minimal cardinality which satisfies the D\"orfler marking. Therefore, we obtain by use of~\eqref{refinement:sons} and~\eqref{refinement:overlay}
\begin{align}\label{eq:marked}
\begin{split}
 \#\MM_\ell\leq \#(\TT_\star\setminus\TT_\ell)\leq \#\TT_\star-\#\TT_\ell\leq \#\TT_\eps - \#\TT_0 
&\lesssim \norm{(u,f)}{\A_s}^{1/s} \eps^{-1/s}\\
&\lesssim \norm{(u,f)}{\A_s}^{1/s}\eta_\ell^{-1/s}
\end{split}
\end{align}
for all $\ell\in\N$.
Finally, the closure estimate~\eqref{refinement:closure} and the contraction~\eqref{eq:rconv} of Proposition~\ref{thm:rconv} yield
\begin{align*}
 \#\TT_\ell-\#\TT_0\lesssim \sum_{j=0}^{\ell-1}\#\MM_\ell\lesssim \norm{(u,f)}{\A_s}^{1/s}\sum_{j=0}^{\ell-1}\eta_j^{-1/s}\lesssim \norm{(u,f)}{\A_s}^{1/s}\eta_\ell^{-1/s}\sum_{j=0}^{\ell-1}\q{rconv}^{(\ell-j)/s}.
\end{align*}
Exploiting the convergence of the geometric series, we end up with
\begin{align*}
 \eta_\ell^2\lesssim \norm{(u,f)}{\A_s} (\#\TT_\ell-\#\TT_0)^{-s}\quad\text{for all }\ell\in\N.
\end{align*}
Altogether, this proves that each theoretically possible convergence rate for the estimator is, in fact, asymptotically achieved by the adaptive algorithm. The converse implication in~\eqref{eq:optimality} is obvious. This concludes the proof.
\end{proof}

\begin{remark}
 We stress that the proof of Theorem~\ref{thm:optimal} depends only on properties~\eqref{refinement:shaperegular}--\eqref{refinement:overlay} of NVB, $R$-linear convergence~\eqref{eq:rconv} of the estimator used, and the discrete reliability~\eqref{eq:drel}. In particular, there is no explicit use of the properties of the differential operator $\operator{L}$, i.e.\ neither linearity nor uniform ellipticity is required.
\end{remark}

\section{Extensions}\label{section:extensions}
In this section, we want to discuss some possible extensions of our analysis.
\subsection{Minimal cardinality of marked elements}
The choice of the set of marked elements $\MM_\ell$ in step~(iii) of Algorithm~\ref{algorithm} to be a set of minimal cardinality which satisfies the D\"orfler marking~\eqref{eq:doerfler}, requires  to sort the set $\set{\eta_\ell(T)}{T\in\TT_\ell}$, which takes at least $\mathcal{O}\big(\#\TT_\ell\log(\#\TT_\ell)\big)$ operations. In comparison to $\mathcal{O}(\#\TT_\ell)$ operations for iterative solvers on sparse matrices, marking becomes the bottleneck of Algorithm~\ref{algorithm}. To overcome this problem, we may allow the set $\MM_\ell$ to be of \emph{almost} minimal cardinality in the sense of
\begin{align}\label{eq:almost}
\#\MM_\ell \leq \c{almost} \#\widetilde\MM_\ell\quad\text{for all }\ell\in\N,
\end{align}
where $\widetilde \MM_\ell $ is a set of minimal cardinality which satisfies D\"orfler marking and $\setc{almost}>0$ is an arbitrary but fixed constant. All the proofs hold true up to an the additional factor $\c{almost}$, which is involved in~\eqref{eq:marked}.  The relaxation~\eqref{eq:almost} allows to apply an inexact sorting algorithm based on binning of the data (see e.g.~\cite{ms}) which performs in $\mathcal{O}(\#\TT_\ell)$ operations.
\subsection{Other mesh-refinement strategies}
Instead of simple \emph{newest-vertex} bisection, one can consider other mesh-refinement strategies which satisfy~\eqref{refinement:sons}--\eqref{refinement:overlay}, since no other property of the mesh refinement strategy is used throughout this paper. In particular, one could use up to $m$ newest vertex bisections per marked element, where $m\in\N$ is a fixed number, cf.\ e.g.~\cite{ks}.
This includes the strategy proposed in~\cite{cn} which uses additional bisections every $n$-th step to ensure the interior node property and hence to obtain a discrete lower bound on the error. Moreover, one can relax the regularity of the triangulations used and allow a fixed number of hanging nodes in each triangle $T\in\TT_\ell$~\cite{bn}.
\subsection{Inhomogeneous Dirichlet data}
Let $\SS^p(\TT_\ell):=\PP^p(\TT_\ell)\cap H^1(\Omega)$ with discrete trace space $\SS^p(\TT_\ell|_\Gamma):=\set{V_\ell|_\Gamma}{V_\ell\in\SS^p(\TT_\ell)}$.
We consider inhomogeneous Dirichlet data $g\in H^{1/2}(\Gamma)$ and an $H^{1/2}$-stable projection $P_\ell:\, H^{1/2}(\Gamma)\to \SS^p(\TT_\ell|_\Gamma)$, for instance the Scott-Zhang projection~\cite{sz} for $p\geq 1$ or the $L^2$-projection for $p=1$ (see~\cite{kpp} for $H^1$-stability on NVB refined meshes). 
The continuous problem we want to solve, now reads: Find $u\in H^1(\Omega)$ with $u|_{\partial \Omega}=g$ such that
\begin{align}\label{eq:continuousinhom}
\dual{\operator{L}u}{v}=b(u,v)=\int_\Omega fv\,dx \quad\text{for all }v\in H^1_0(\Omega).
\end{align}
 The corresponding discrete formulation reads: Find $U_\ell \in\SS^p(\TT_\ell)$ with $U_\ell|_\Gamma=P_\ell g$ such that
\begin{align}\label{eq:discreteinhom}
b(U_\ell,V_\ell)=\int_\Omega f V_\ell\,dx \quad\text{for all }V_\ell\in\SS^p_0(\TT_\ell).
\end{align}
Well-posedness of~\eqref{eq:continuousinhom}--\eqref{eq:discreteinhom} is well-known and discussed, e.g., in~\cite{dirichlet3d,bcd,sv}.
The approximation error which is introduced via $g\approx P_\ell g$ results in an additional error quantity. We assume regularity $g\in H^1(\Gamma)$ and define the Dirichlet data oscillations
\begin{align*}
\osc_{g,\ell}:=\sum_{E\in\TT_\ell|_\Gamma} \text{diam}(E)\normLtwo{\nabla_{\Gamma}(1-P_\ell)g}{E}^2,
\end{align*}
where $\nabla_{\Gamma}(\,\cdot\,)$ denotes the surface gradient on $\Gamma=\partial\Omega$.

Since the ansatz spaces are no longer nested, i.e. $U_{\ell+1}-U_\ell\notin \SS^p_0(\TT_\ell)$, we have to rely on a modified marking strategy proposed in~\cite{stevenson07}. We replace the D\"orfler marking~\eqref{eq:doerfler}
by the following separate marking strategy with adaptivity parameters $0<\theta,\vartheta<1$:
\begin{itemize}
\item If $\osc_{g,\ell}^2\leq\vartheta \eta_\ell^2$, determine $\MM_\ell\subseteq \TT_\ell$ as a set of minimal cardinality which satisfies~\eqref{eq:doerfler}.
\item If $\osc_{g,\ell}^2>\vartheta \eta_\ell^2$, determine $\MM_\ell\subseteq \TT_\ell$ as a set of minimal cardinality which satisfies
\begin{align}\label{eq:sdoerfler2}
\theta \osc_{g,\ell}^2\leq \sum_{T\in\MM_\ell} \osc_{g,\ell}(T)^2.
\end{align}
\end{itemize}
Now, the analysis of~\cite{dirichlet3d} can easily be transfered to the present problem as well, where $\eta_\ell$ in~\eqref{eq:estred},~\eqref{eq:rconv}, and~\eqref{eq:approxclass}--\eqref{eq:optimality} is replaced by $\rho_\ell:=\eta_\ell+\osc_{g,\ell}$. For usual choices of $P_\ell$ as above, one obtains convergence of AFEM by means of the estimator reduction principle~\cite[Theorem~4]{dirichlet3d}. Moreover, for arbitrary $P_\ell$ and sufficiently small marking parameters $0<\vartheta,\theta<1$, we obtain the optimality result of Theorem~\ref{thm:optimal}, cf.~\cite[Theorem~6]{dirichlet3d}.

For $d=2$, one may even use nodal interpolation to discretize the inhomogeneous Dirichlet data. Then, the combined D\"orfler marking~\eqref{eq:doerfler} for $\rho_\ell:=\eta_\ell+\osc_{g,\ell}$ instead of $\eta_\ell$ yields the contraction result of Theorem~\ref{thm:rconv}. Moreover, for sufficiently small $0<\theta<1$, Theorem~\ref{thm:optimal} remains valid. We refer to~\cite{fpp} in case of symmetric $\operator{L}=-\Delta$ and stress that the analysis can easily be transfered to the present setting.

\subsection{Coercive but not uniformly elliptic bilinear forms}
Assume that instead of ellipticity~\eqref{eq:elliptic}, there holds a G\r{a}rding inequality
\begin{align}\label{eq:garding}
b(u,u)+\c{garding}\normLtwo{u}{\Omega}^2\geq \rr{garding}\normLtwo{\nabla u}{\Omega}^2\quad\text{for all }u \in H^1(\Omega)
\end{align}
with constants $0<\setrr{garding}<1$ and $\setc{garding}>0$
We have to assume that $b(\cdot,\cdot)$ is definite on the continuous level, i.e. for all $v\in H^1_0(\Omega)$
, it holds 
\begin{subequations}\label{eq:injective}
\begin{align}
 b(v,w)&=0\quad\text{for all } w\in H^1_0(\Omega) \quad\implies\quad v=0,\\
\label{eq:injectiveinfty} b(v_\infty,w_\infty)&=0\quad\text{for all } w_\infty\in \SS_0^p(\TT_\infty) \quad\implies\quad v_\infty=0.
\end{align}
\end{subequations}
This together with Fredholm's alternative already guarantees the unique solvability of~\eqref{eq:continuous} and~\eqref{eq:discrete} with test and ansatz space  $\SS^p_0(\TT_\infty)$ instead of $\SS^p_0(\TT_\ell)$. 

\begin{remark}
 Usually, the conditions~\eqref{eq:injective} are guaranteed under the assumption that the mesh-size of the initial mesh $\TT_0$ is sufficiently small and that the solution $w\in H^1_0(\Omega)$ of the dual problem
\begin{align*}
 b(v,w)=\int_\Omega f v\,dx\quad\text{for all }v\in H^1_0(\Omega)
\end{align*}
satisfies  some regularity estimate
\begin{align*}
 \norm{w}{H^{1+s}(\Omega)}\lesssim \norm{f}{L^2(\Omega)}\quad\text{for some }s>0,
\end{align*}
see e.g.~\cite[Theorem~5.7.6]{bs}.
\end{remark}

Now, we may apply~\cite[Theorem~4.2.9]{ss} to obtain the following result.
\begin{lemma}
 There exists an index $\ell_0\in\N$ such that for all $\ell\geq\ell_0$ the discrete formulation~\eqref{eq:discrete} is uniquely solvable, and it holds
\begin{align}\label{eq:ceagarding}
 \normLtwo{\nabla(u_\infty -U_\ell)}{\Omega}\leq \c{cea}\min_{V_\ell\in\SS^p_0(\TT_\ell)}\normLtwo{\nabla(u_\infty-V_\ell)}{\Omega},
\end{align}
where $u_\infty\in\SS^p_0(\TT_\infty)$ denotes the unique solution of~\eqref{eq:discrete} with $\SS^p_0(\TT_\infty)$ instead of $\SS^p_0(\TT_\ell)$.
\end{lemma}
\begin{proof}
Since~\eqref{eq:garding} states that $b(u,v)+ \c{garding}\dual{u}{v}_{L^2(\Omega)}$ is elliptic and $\dual{\cdot}{\cdot}_{L^2(\Omega)}$ is a compact perturbation, we apply~\cite[Theorem~4.2.9]{ss} on the Hilbert space $\SS^p_0(\TT_\infty)$ and the dense sequence of subspaces $\SS^p_0(\TT_\ell)$ for $\ell\to\infty$.
\end{proof}

The above lemma allows to prove a~priori convergence from Lemma~\ref{lem:apriori} and consequently convergence $U_\ell\to u$ in $H^1_0(\Omega)$ as well as $u\in\SS^p_0(\TT_\infty)$. Moreover, Lemma~\ref{lem:weakconv} still holds true, since we assumed definiteness of $b(\cdot,\cdot)$ on $\SS^p_0(\TT_\infty)$ in~\eqref{eq:injectiveinfty}.
\begin{lemma}\label{lem:ellipticgarding}
 There exists an index $\ell_1\in\N$ such that for all $\ell\geq\ell_1$ there holds
\begin{align*}
 \normLtwo{\nabla(u-U_\ell)}{\Omega}\leq \c{norm}\enorm{u-U_\ell} \quad\text{and}\quad\normLtwo{\nabla(U_{\ell+1}-U_\ell)}{\Omega}\leq \c{norm}\enorm{U_{\ell+1}-U_\ell}. 
\end{align*}
\end{lemma}
\begin{proof}
With~\eqref{eq:garding} and $b(\cdot,\cdot)=\enorm{\cdot}^2$, we may estimate
\begin{align*}
  \normLtwo{\nabla(u-U_{\ell})}{\Omega}^2&\lesssim \enorm{u-U_\ell}^2 +\normLtwo{u-U_{\ell}}{\Omega}^2\\
&=\enorm{u-U_\ell}^2 +\normLtwo{e_\ell}{\Omega}^2\normLtwo{\nabla(u-U_\ell)}{\Omega}^2.
\end{align*}
Lemma~\ref{lem:weakconv} shows weak convergence $e_\ell\rightharpoonup 0$ in $H^1_0(\Omega)$. The Rellich compactness theorem thus implies strong convergence $e_\ell\to 0$ in $L^2(\Omega)$. Therefore, there exists an index $\ell_1\in\N$ such that
 there holds
\begin{align*}
  \normLtwo{\nabla(u-U_{\ell})}{\Omega}^2\lesssim \enorm{u-U_\ell}^2\quad\text{for all }\ell\geq\ell_1.
\end{align*}
The statement for $U_{\ell+1}-U_\ell$ follows analogously.
\end{proof}
Lemma~\ref{lem:weakconv} together with Lemma~\ref{lem:ellipticgarding} allows to prove the quasi-Galerkin orthogonality of Proposition~\ref{prop:quasiqo} and consequently also the $R$-linear convergence of Theorem~\ref{thm:rconv}. Therefore, all the results from Section~\ref{section:optimality} hold and, in particular, we obtain the optimality result of Theorem~\ref{thm:optimal}.

\subsection{Non-linear operators $\operator{L}$}\label{section:nonlin}
We consider the following \emph{non-linear} operator
\begin{align*}
 \operator{L}u(x):= -\text{div} \matrix{A}(x,\nabla u(x)) + g(x,u(x),\nabla u(x)),
\end{align*}
for functions $\matrix{A}:\,\Omega\times\R^d \to \R^d$ and $g:\, \Omega\times\R\times \R^d \to \R$. We assume that $\matrix{A}(\cdot,\nabla u),g(\cdot,u,\nabla u)\in L^2(\Omega)$ for all $u\in H^1_0(\Omega)$. Then, the weak formulation of~\eqref{intro:modelproblem} reads: Find $u\in H^1_0(\Omega)$ such that
\begin{align}\label{eq:continuousnonlin}
 \dual{\operator{L}u}{v}=\int_\Omega \matrix{A}(x,\nabla u(x))\cdot\nabla v(x) + g(x,u(x),\nabla u(x)) v(x)\,dx = \int_\Omega fv\,dx
\end{align}
for all $v\in H^1_0(\Omega)$. We define two auxiliary operators $\operator{A},\operator{K}:\,H^1_0(\Omega)\to H^{-1}(\Omega)$ as
\begin{align*}
 \operator{A}v:=-\text{div}\matrix{A}(\cdot,\nabla v)\quad\text{and}\quad\operator{K}v:=g(\cdot,v,\nabla v)\quad\text{for all }v\in H^1_0(\Omega).
\end{align*}
We formally define the residual error estimator for a mesh $\TT_\ell$
\begin{align}\label{eq:nlest}
 \eta_\ell^2:= \sum_{T\in\TT_\ell}\big(|T|^{2/d}\normLtwo{\operator{L}|_TU_\ell - f}{T}^2 + |T|^{1/d} \normLtwo{[\matrix{A}(\cdot,\nabla U_\ell)\cdot n]}{\partial T\cap\Omega}^2\big).
\end{align}
The solvability and uniqueness of~\eqref{eq:continuousnonlin} as well as the regularity assumptions needed such that~\eqref{eq:nlest} is well-defined are part of the subsequent sections.
\subsubsection{Regularity assumptions}\label{section:nlreg}
 We consider the frame of strongly monotone operators and require the following regularity assumptions on $\operator{L}$:
\begin{subequations}\label{eq:smon1}
\begin{align}\label{eq:smon1a}
\norm{\operator{A}\nabla w-\operator{A}\nabla v}{H^{-1}(\Omega)}&\leq \c{nllip}\normLtwo{\nabla(w-v)}{\Omega},\\
\normLtwo{\operator{K}w-\operator{K}v}{\Omega}&\leq \c{nllip}\normLtwo{\nabla(w-v)}{\Omega}\label{eq:smon1b}
\end{align}
\end{subequations}
for all $w,v\in H^1_0(\Omega)$ and some constant $\setc{nllip}>0$ as well as
\begin{align}\label{eq:smon2}
 \dual{\operator{L}w-\operator{L}v}{u-v}\geq \c{nlelliptic} \normLtwo{\nabla(w-v)}{\Omega}^2
\end{align}
for all $w,v\in H^1_0(\Omega)$ and some constant $\setc{nlelliptic}>0$. 
These assumptions, in particular, allow to apply the main theorem on strongly monotone operators~\cite[Theorem~26.A]{zeidler} and to obtain the unique solvability of~\eqref{eq:continuousnonlin} as well as of~\eqref{eq:discrete}. Additionally,~\eqref{eq:smon1}--\eqref{eq:smon2} guarantee that the norms of the residual and the error are equivalent, i.e. \begin{align}\label{eq:resequiv}
\normHme{\operator{L}u-\operator{L}U_\ell}{\Omega}\simeq \normLtwo{\nabla(u-U_\ell)}{\Omega}\quad\text{for all }\ell\in\N.
\end{align}
We also obtain the C\'ea lemma~\eqref{eq:cea} with the constant $2\c{nllip}/\c{nlelliptic}$.

Moreover, we require that~\eqref{eq:nlest} is well-defined and that there holds the estimator reduction~\eqref{eq:estred} from Lemma~\ref{lem:estred}. For possible non-linearities $\matrix{A}$ which allow for~\eqref{eq:estred}, we refer to Lemma~\ref{lem:nlwell} below.

We assume that $\operator{L}:\, H_0^1(\Omega)\to H^{-1}(\Omega)$ as well as $\operator{A}:\, H_0^1(\Omega)\to H^{-1}(\Omega)$ are twice Fr\'echet differentiable, i.e. there exist
\begin{align}\label{eq:frechet}
 \begin{split}
  D\operator{L},D\operator{A}:&\, H^1_0(\Omega)\to L(H^1_0(\Omega),H^{-1}(\Omega)),\\
  D^2\operator{L},D^2\operator{A}:&\, H^1_0(\Omega)\to L\big(H^1_0(\Omega),L(H^1_0(\Omega),H^{-1}(\Omega))\big).
 \end{split}
\end{align}
The second derivative should be bounded locally around the solution $u$ of~\eqref{eq:continuousnonlin} i.e., there exists $\eps_{\ell oc}>0$ with \definec{nlbound}
\begin{align}\label{eq:nlbounded}
\begin{split}
 \c{nlbound}:=\sup_{\normLtwo{\nabla(u-v)}{\Omega}<\eps_{\ell oc}}\Big(\norm{&D^2\operator{L}(v)}{L\big(H^1_0(\Omega),L(H^1_0(\Omega),H^{-1}(\Omega))\big)}\\
&+\norm{D^2\operator{A}(v)}{L\big(H^1_0(\Omega),L(H^1_0(\Omega),H^{-1}(\Omega))\big)}\Big)<\infty.
\end{split}
\end{align}
Finally, we assume that $D\operator{A}(v):\,H^1_0(\Omega)\to H^{-1}(\Omega)$ is symmetric for all $v\in H^1_0(\Omega)$, i.e.\ for all $w_1,w_2\in H^1_0(\Omega)$ holds
\begin{align*}
\dual{D\operator{A}(v)(w_1)}{w_2}=\dual{D\operator{A}(v)(w_2)}{w_1}.
\end{align*}

\begin{remark}
Note that if $\matrix{A}:\,\Omega\times\R^d \to \R^d$ and $g:\, \Omega\times\R\times \R^d \to \R$ are twice differentiable, and if the Jacobian
$J_y\matrix{A}(x,y)\in \R^{d\times d}$ additionally is a symmetric matrix, then $\operator{L}$ and $\operator{A}$ satisfy~\eqref{eq:frechet} as well as~\eqref{eq:nlbounded}. Moreover, $D\operator{A}(v)$ is symmetric for all $v\in H^1_0(\Omega)$, since there holds for $w\in H^1_0(\Omega)$
\begin{align*}
 D\operator{A}(v)(w)={\rm div}\Big(\big(J_y\matrix{A}(\cdot,\nabla v(\cdot))\big) \big(\nabla w(\cdot)\big)\Big),
\end{align*}
where $J_y\matrix{A}(x,y)$ denotes the Jacobian of $\matrix{A}$ with respect to $y$.
\end{remark}

\begin{example}
 We stress that the assumptions on $\operator{A}$ and $\operator{L}$ posed, cover  for instance non-linear material laws in magnetostatics, where e.g.\ $\matrix{A}(\cdot,\cdot)$ takes the form 
\begin{align*}
\matrix{A}(x,\nabla u(x)) = \Big(1+\frac{1}{1+|\nabla u(x)|^2}\Big)\nabla u(x).
\end{align*}
E.g.\ for $d=2$, the Jacobi-matrix $J_y\matrix{A}(x,y)$ reads as
\begin{align*}
 J_y\matrix{A}(x,y):=\left(\begin{array}{cc}
                                 \frac{-2y_1^2}{(1+|y|^2)^2} & \frac{-2y_1y_2}{(1+|y|^2)^2}\\
				  \frac{-2y_1y_2}{(1+|y|^2)^2} &\frac{-2y_2^2}{(1+|y|^2)^2}
                                \end{array}\right)
			      +
				 \left(\begin{array}{cc}
                                 1+\frac{1}{1+|y|^2} & 0\\
				  0 &1+\frac{1}{1+|y|^2}
                                \end{array}\right).
\end{align*}
We refer to e.g.~\cite{rzmp} for further examples.
\end{example}

\begin{lemma}\label{lem:nlwell}
Sufficient regularity assumptions in addition to~\eqref{eq:smon1b} and~\eqref{eq:smon2} to guarantee that the error estimator~\eqref{eq:nlest} is well-defined and satisfies the estimator reduction~\eqref{eq:estred} are, for instance, either of the following conditions $(i)$ and $(ii)$:
\begin{itemize}
 \item[(i)] $\matrix{A}(\cdot,\cdot):\,\Omega\times\R^d\to\R^d$ is Lipschitz continuous and there exists a constant $\setc{nlwell1}>0$ such that for all $\ell\in\N$ and all $V_\ell,W_\ell\in\SS^p_0(\TT_\ell)$ there holds ${\rm div}\matrix{A}(\cdot, V_\ell(\cdot))\in L^2(\Omega)$ as well as
 \begin{align}\label{eq:nlwell1}
 \normLtwo{{\rm div}|_T\big(\matrix{A}(\cdot,V_\ell(\cdot))-\matrix{A}(\cdot, W_\ell(\cdot))\big)}{T}\leq\c{nlwell1}\norm{V_\ell-W_\ell}{H^2(T)}\quad\text{for all }T\in\TT_\ell.
 \end{align}
\item[(ii)] There holds $p=1$ (lowest-order case) as well as
\begin{align*}
 \matrix{A}(x,y)=\matrix{A}(y)\quad\text{for all }x\in\Omega,\,y\in \R^d,
\end{align*}
and additionally  $\matrix{A}(\cdot):\,\R^d\to\R^d$ is Lipschitz continuous.
\end{itemize}
\end{lemma}
\begin{proof}
The jump terms in~\eqref{eq:nlest} are well-defined in both cases $(i)$ and $(ii)$ since $\matrix{A}(\cdot,\nabla U_\ell(\cdot))$ is a piecewise Lipschitz continuous function. Moreover, this shows that ${\rm div}\matrix{A}(\cdot,\nabla U_\ell(\cdot))\in L^\infty(T)\subset L^2(T)$ for all $T\in\TT_\ell$. Therefore,~\eqref{eq:nlest} is well-defined.

Given $T_+,T_-\in\TT_\ell$ as well as $W_\ell,V_\ell\in\SS^p_0(\TT_\ell)$, the Lipschitz continuity also proves the following pointwise estimate for all $x\in T_+\cap T_-$
\begin{align*}
 |[(\matrix{A}(x,\nabla W_\ell(x))&-\matrix{A}(x,\nabla V_\ell(x)))\cdot n]|\\
 &\lesssim \Big|\big(\matrix{A}(x,(\nabla W_\ell)|_{T_+}(x))-\matrix{A}(x,(\nabla V_\ell)|_{T_+}(x))\big)\cdot n|_{T_+}\\
 &\qquad+
 \big(\matrix{A}(x,(\nabla W_\ell)|_{T_-}(x))-\matrix{A}(x,(\nabla V_\ell)|_{T_-}(x))\big)\cdot n|_{T_-}\Big|\\
 &\lesssim \Big|(\nabla W_\ell)|_{T_+}(x)-(\nabla V_\ell(x))|_{T_+} \Big|+ \Big|(\nabla W_\ell)|_{T_-}(x)-(\nabla V_\ell)|_{T_-}(x) \Big|.
\end{align*}
Combining the estimate above with the trace inequality for polynomials, we obtain
\begin{align}\label{eq:nlwellhelp1}
 |T_+|^{1/d} \normLtwo{[(\matrix{A}(\cdot,\nabla W_\ell)-\matrix{A}(\cdot,\nabla V_\ell))\cdot n]}{T_+\cap T_-}^2\lesssim 
\normLtwo{\nabla(W_\ell-V_\ell)}{T_+\cup T_-}^2.
\end{align}
This hidden constant depends only on the polynomial degree $p\in\N$ as well as the Lipschitz continuity of $\matrix{A}(\cdot,\cdot)$ and the $\gamma$-shape regularity of $\TT_\ell$.
It remains to prove a similar estimate for the volume residual in~\eqref{eq:nlest}, i.e.
\begin{align}\label{eq:nlwellhelp2}
 |T|^{2/d}\normLtwo{\operator{L}|_TW_\ell - \operator{L}|_T V_\ell}{T}^2\lesssim \normLtwo{\nabla(W_\ell-V_\ell)}{T}^2\quad\text{for all }T\in\TT_\ell.
\end{align}
In case of $(i)$, this follows immediately from the combination of~\eqref{eq:nlwell1} and~\eqref{eq:smon1b} together with a standard inverse estimate.
In case of $(ii)$, we observe that $\nabla U_\ell$ is piecewise constant. Therefore, $\matrix{A}(\nabla U_\ell)$ is also piecewise constant and hence $\operator{A}(\nabla U)={\rm div}\matrix{A}(\nabla U(\cdot)) = 0$. Thus, $\operator{L}|_TV_\ell=(\operator{K}V_\ell)|_T$, and it suffices to apply~\eqref{eq:smon1b} to prove~\eqref{eq:nlwellhelp2}. 
With the estimates~\eqref{eq:nlwellhelp1}--\eqref{eq:nlwellhelp2}, the proof of Lemma~\ref{lem:estred} still holds true with the obvious modifications. This concludes the proof.
\end{proof}

\subsubsection{Auxiliary results}
This section provides some technical lemmata, which are used to transfer the results from the linear case to the present non-linear case.
\begin{lemma}\label{lem:nlrel}
The residual error estimator satisfies reliability~\eqref{eq:reliable} as well as discrete reliability~\eqref{eq:drel}.
Moreover, there holds convergence
\begin{align}\label{eq:nlconv}
 \normLtwo{\nabla(u-U_\ell)}{\Omega}\to 0\quad\text{as }\ell\to\infty.
\end{align}
\end{lemma}
\begin{proof}
The residual error estimator $\eta_\ell$ is well-defined by assumption in Section~\ref{section:nlreg}. With the equivalence~\eqref{eq:resequiv}, the standard arguments apply to prove reliability~\eqref{eq:reliable} and also the proof of discrete reliability~\eqref{eq:drel} follows analogously to~\cite{ckns}.
The estimator reduction holds by assumption in Section~\ref{section:nlreg} and therefore Proposition~\ref{prop:conv} holds true and proves~\eqref{eq:nlconv}.
\end{proof}

\begin{lemma}\label{lem:aux1} 
The operator $(D\operator{L})|_{\SS^p_0(\TT_\infty)}u:\,\SS^p_0(\TT_\infty)\to \SS^p_0(\TT_\infty)^\star$ is injective.
\end{lemma}
\begin{proof}
With~\eqref{eq:smon2} and the definition of the Fr\'echet derivative, there holds for all $v\in \SS^p_0(\TT_\infty)$ with $\normLtwo{\nabla v}{\Omega}=1$
\begin{align*}
 \dual{((D\operator{L})|_{\SS^p_0(\TT_\infty)}u)(v)}{v}&=\lim_{\delta\to 0} \delta^{-2} \dual{\operator{L}(u+\delta v)-\operator{L}u}{u+\delta v- u}\\
&\gtrsim \delta^{-2}\normLtwo{\nabla(u+\delta v-u)}{\Omega}^2 =1.
\end{align*}
Hence, we have $((D\operator{L})|_{\SS^p_0(\TT_\infty)}u)(v)\neq 0$ in $\SS^p_0(\TT_\infty)^\star$ for all $v\in \SS^p_0(\TT_\infty)\setminus\{0\}$. This concludes the proof.
\end{proof}
\begin{lemma}[Taylor]\label{lem:taylor}
For all $v,w\in H^1_0(\Omega)$ with $\normLtwo{\nabla(u-v)}{\Omega}+\normLtwo{\nabla(u-w)}{\Omega}\leq \eps_{\ell oc}$, there holds
\begin{subequations}\label{eq:taylor}
\begin{align}\label{eq:taylor1}
\normHme{\operator{L}w-\operator{L}v- D\operator{L}(w-v)}{\Omega}&\leq \c{nlbound}\normLtwo{\nabla(w-v)}{\Omega}^2,\\
\normHme{\operator{A}w-\operator{A}v- D\operator{A}(w-v)}{\Omega}&\leq \c{nlbound}\normLtwo{\nabla(w-v)}{\Omega}^2.\label{eq:taylor2}
\end{align}
\end{subequations}
\end{lemma}
\begin{proof}
The local boundedness~\eqref{eq:nlbounded} together with~\cite[Theorem 6.5]{cp} applied to the operators $\operator{L}$ and $\operator{A}$ proves the statement.
\end{proof}
\subsubsection{Quasi-orthogonality}
Following the steps of Section~\ref{section:quasiqo}, we derive a similar result for the present, non-linear case.
\begin{lemma}\label{lem:nlweakconv}
The sequence $(e_\ell)_{\ell\in\N}$ defined by
\begin{align*}
e_\ell:=\begin{cases}\frac{u-U_\ell}{\normLtwo{\nabla(u-U_\ell)}{\Omega}},& \text{ for }u\neq U_\ell,\\
	0, &\text{ else}\end{cases}
\end{align*}
converges to zero, weakly in $ H^1_0(\Omega)$.
\end{lemma}
\begin{proof}
With Galerkin-orthogonality and the convention $\infty\cdot 0 = 0$, we obtain
\begin{align*}
 \lim_{\ell\to\infty}\frac{\dual{\operator{L}u-\operator{L}U_\ell}{V_k}}{\normLtwo{\nabla(u-U_\ell)}{\Omega}} = 0\quad\text{for all }V_k\in \SS^p_0(\TT_k)\text{ and }k\in\N.
\end{align*}
By continuity of the duality brackets, this results in convergence for all $v\in\SS^p_0(\TT_\infty)$
\begin{align*}
\frac{\dual{\operator{L}u-\operator{L}U_\ell}{v}}{\normLtwo{\nabla(u-U_\ell)}{\Omega}}\to 0\quad\text{as }\ell\to\infty.
\end{align*}
By use of~\eqref{eq:taylor1}, we observe for all $v\in\SS^p_0(\TT_\infty)$
\begin{align*}
\frac{|\dual{\operator{L}u-\operator{L}U_\ell}{v}|}{\normLtwo{\nabla(u-U_\ell)}{\Omega}}&\geq \frac{|\dual{(D\operator{L}u)(u-U_\ell)}{v}|}{\normLtwo{\nabla(u-U_\ell)}{\Omega}}
-\c{nlbound}\normLtwo{\nabla(u-U_\ell)}{\Omega}\normLtwo{\nabla v}{\Omega}.
\end{align*}
With convergence $U_\ell \to u$ in $H^1_0(\Omega)$ from~\eqref{eq:nlconv}, this implies immediately
\begin{align}\label{eq:nlhelp}
 \frac{|\dual{u-U_\ell}{((D\operator{L})|_{\SS^p_0(\TT_\infty)}u)^\star v}|}{\normLtwo{\nabla(u-U_\ell)}{\Omega}}\to 0\quad\text{as }\ell\to\infty\quad\text{for all }v\in\SS^p_0(\TT_\infty).
\end{align}
According to Lemma~\ref{lem:aux1}, $(D\operator{L})|_{\SS^p_0(\TT_\infty)}u$ is injective. Therefore, its adjoint $((D\operator{L})|_{\SS^p_0(\TT_\infty)}u)^\star$ is surjective onto $\SS^p_0(\TT_\infty)^\star$. Hence,~\eqref{eq:nlhelp} is equivalent to $e_\ell\rightharpoonup 0$ as $\ell\to\infty$. This concludes the proof.
\end{proof}
To abbreviate notation, we define the quasi-metric
\begin{align*}
 \myd(w,v)^2:=\dual{\operator{L}w-\operator{L}v}{w-v}\quad\text{for all }w,v\in H^1_0(\Omega).
\end{align*}
Note that due to~\eqref{eq:smon1}--\eqref{eq:smon2}, there holds
\begin{align}\label{eq:dequiv}
 \c{norm}^{-1}\normLtwo{\nabla(w-v)}{\Omega}\leq \myd(w,v)\leq\c{norm}\normLtwo{\nabla(w-v)}{\Omega}\quad\text{for all } w,v\in H^1_0(\Omega)
\end{align}
with $\c{norm}=\max\{2\c{nllip},\c{nlelliptic}^{-1}\}>0$.

\begin{proposition}\label{prop:nlquasiqo}
For any $\eps>0$, there exists $\ell_0\in\N$ such that
\begin{align}\label{eq:nlquasiqo}
\myd(U_{\ell+1},U_\ell)^2\leq \frac{1}{1-\eps}\,\myd(u,U_\ell)^2-\myd(u,U_{\ell+1})^2
\end{align}
for all $\ell\geq \ell_0$.
\end{proposition}
\begin{proof}
Due to convergence $U_\ell\to u$ in $H^1_0(\Omega)$~\eqref{eq:nlconv}, there exists $\ell_1\in\N$ such that for all $\ell\geq\ell_1$ we may apply~\eqref{eq:taylor2}, to obtain
\begin{align*}
 |\dual{\operator{A}U_{\ell+1}-\operator{A}U_\ell}{u-U_{\ell+1}}|&\leq |\dual{D\operator{A}(U_{\ell+1})(U_{\ell+1}-U_\ell)}{u-U_{\ell+1}}|\\
&\qquad\qquad+\c{nlbound}\normLtwo{\nabla(U_{\ell+1}-U_\ell)}{\Omega}^2\normLtwo{\nabla(u-U_{\ell+1})}{\Omega}
\end{align*}
Using the symmetry of $D\operator{A}(U_{\ell+1})$, we conclude
\begin{align*}
  |\dual{\operator{A}U_{\ell+1}-\operator{A}U_\ell}{u-U_{\ell+1}}|&\leq |\dual{D\operator{A}(U_{\ell+1})(u-U_{\ell+1})}{U_{\ell+1}-U_\ell}|\\
&\qquad\qquad+\c{nlbound}\normLtwo{\nabla(U_{\ell+1}-U_\ell)}{\Omega}^2\normLtwo{\nabla(u-U_{\ell+1})}{\Omega}\\
&\leq |\dual{\operator{A}u-\operator{A}U_{\ell+1}}{U_{\ell+1}-U_\ell}|\\
&\qquad\qquad+\c{nlbound}\normLtwo{\nabla(U_{\ell+1}-U_\ell)}{\Omega}^2\normLtwo{\nabla(u-U_{\ell+1})}{\Omega}\\
&\qquad\qquad+\c{nlbound}\normLtwo{\nabla(U_{\ell+1}-U_\ell)}{\Omega}\normLtwo{\nabla(u-U_{\ell+1})}{\Omega}^2.
\end{align*}
Analogously to the estimate above, we obtain a lower estimate. For any $\delta>0$, we may thus use convergence $U_\ell\to u$ as $\ell\to\infty$ to find an index $\ell_0\in \N$ such that
\begin{align*}
  \big||\dual{\operator{A}U_{\ell+1}-\operator{A}U_\ell}{u-U_{\ell+1}}|&-|\dual{\operator{A}u-\operator{A}U_{\ell+1}}{U_{\ell+1}-U_\ell}|\big|\\
&\qquad\qquad\leq\delta\normLtwo{\nabla(U_{\ell+1}-U_\ell)}{\Omega}\normLtwo{\nabla(u-U_{\ell+1})}{\Omega}
\end{align*}
 for all $\ell\geq\ell_0$.
Since $e_\ell$ converges to zero weakly in $H^1_0(\Omega)$, we have strong convergence $e_\ell\to 0$ as $\ell\to\infty$ in $L^2(\Omega)$. This together with Lipschitz continuity~\eqref{eq:smon1b} allows to estimate
\begin{align*}
 |\dual{\operator{K}U_{\ell+1}-\operator{K}U_\ell}{u-U_{\ell+1}}|\lesssim \normLtwo{\nabla(U_{\ell+1}-U_\ell)}{\Omega}\normLtwo{e_{\ell+1}}{\Omega}\normLtwo{\nabla(u-U_{\ell+1})}{\Omega}
\end{align*}
and hence
\begin{align*}
| \dual{\operator{K}U_{\ell+1}-\operator{K}U_\ell}{u-U_{\ell+1}}|\leq \delta \normLtwo{\nabla(U_{\ell+1}-U_\ell)}{\Omega}\normLtwo{\nabla(u-U_{\ell+1})}{\Omega}
\end{align*}
for all $\ell\geq\ell_1$. The adjoint term follows analogously, since
\begin{align*}
 |\dual{\operator{K}u-\operator{K}U_{\ell+1}}{U_{\ell+1}-U_\ell}|\leq |\dual{\operator{K}u-\operator{K}U_{\ell+1}}{U_{\ell+1}-u}|+|\dual{\operator{K}u-\operator{K}U_{\ell+1}}{u-U_\ell}|.
\end{align*}
So far, we end up with
\begin{align*}
 | \dual{\operator{K}U_{\ell+1}-\operator{K}U_\ell}{u-U_{\ell+1}}|&+|\dual{\operator{K}u-\operator{K}U_{\ell+1}}{U_{\ell+1}-U_\ell}|\\
&\leq \delta \big(\normLtwo{\nabla(U_{\ell+1}-U_\ell)}{\Omega}\normLtwo{\nabla(u-U_{\ell+1})}{\Omega}\\
&\qquad +\normLtwo{\nabla(u-U_{\ell+1})}{\Omega}^2 \\
&\qquad + \normLtwo{\nabla(u-U_{\ell+1})}{\Omega}\normLtwo{\nabla(u-U_\ell)}{\Omega}\big)\\
&\leq \delta/2\normLtwo{\nabla(U_{\ell+1}-U_\ell)}{\Omega}^2+2\delta\normLtwo{\nabla(u-U_{\ell+1})}{\Omega}^2\\
&\qquad\qquad+\delta/2\normLtwo{\nabla(u-U_\ell)}{\Omega}^2
\end{align*}
by use of Young's inequality. Putting everything together, we obtain
\begin{align*}
 |\dual{(\operator{A}+\operator{K})U_{\ell+1}&-(\operator{A}+\operator{K})U_\ell}{u-U_{\ell+1}}|\\
&\leq
|\dual{\operator{A}u-\operator{A}U_{\ell+1}}{U_{\ell+1}-U_\ell}|
+\delta\normLtwo{\nabla(U_{\ell+1}-U_\ell)}{\Omega}\normLtwo{\nabla(u-U_{\ell+1})}{\Omega}\\
&\qquad+|\dual{\operator{K}U_{\ell+1}-\operator{K}U_\ell}{u-U_{\ell+1}}|\\
&\leq |\dual{(\operator{A}+\operator{K})u-(\operator{A}+\operator{K})U_{\ell+1}}{U_{\ell+1}-U_\ell}|\\
&\qquad+\delta\normLtwo{\nabla(U_{\ell+1}-U_\ell)}{\Omega}\normLtwo{\nabla(u-U_{\ell+1})}{\Omega}\\
&\qquad+|\dual{\operator{K}U_{\ell+1}-\operator{K}U_\ell}{u-U_{\ell+1}}|
+|\dual{\operator{K}u-\operator{K}U_{\ell+1}}{U_{\ell+1}-U_\ell}|\\
&\leq 3\delta\big(\normLtwo{\nabla(U_{\ell+1}-U_\ell)}{\Omega}^2
+\normLtwo{\nabla(u-U_{\ell+1})}{\Omega}^2
+\normLtwo{\nabla(u-U_\ell)}{\Omega}^2\big),
\end{align*}
where we used Galerkin orthogonality $\dual{(\operator{A}+\operator{K})u-(\operator{A}+\operator{K})U_{\ell+1}}{U_{\ell+1}-U_\ell}=0$ to obtain the last estimate.
With that at hand, we obtain similarly to~\eqref{eq:key}
\begin{align*}
 \myd(U_{\ell+1},U_\ell)^2&\leq  \myd(u,U_\ell)^2 - \myd(u,U_{\ell+1})^2 + |\dual{(\operator{A}+\operator{K})U_{\ell+1}-(\operator{A}+\operator{K})U_\ell}{u-U_{\ell+1}}|\\
&\leq \myd(u,U_\ell)^2 - \myd(u,U_{\ell+1})^2 +3\delta\big(\normLtwo{\nabla(U_{\ell+1}-U_\ell)}{\Omega}^2\\
&\qquad\qquad+\normLtwo{\nabla(u-U_{\ell+1})}{\Omega}^2
+\normLtwo{\nabla(u-U_\ell)}{\Omega}^2\big).
\end{align*}
With the equivalence~\eqref{eq:dequiv}, we conclude
\begin{align*}
 (1-3\c{norm}\delta)\myd(U_{\ell+1},U_\ell)^2&\leq  (1+3\c{norm}\delta)\myd(u,U_\ell)^2 - (1-3\c{norm}\delta)\myd(u,U_{\ell+1})^2
\end{align*}
for all $\ell\geq\ell_0$. Finally, we choose $\delta>0$ sufficiently small such that $ (1+3\c{norm}\delta)/(1-3\c{norm}\delta)\leq 1/(1-\eps)$ and conclude the proof.
\end{proof}
Together with the estimator reduction~\eqref{eq:estred} which holds by assumption in Section~\ref{section:nlreg}, the quasi-Galerkin orthogonality~\eqref{eq:nlquasiqo} of Proposition~\ref{prop:nlquasiqo} allows to prove the $R$-linear convergence of Theorem~\ref{thm:rconv}, if one exchanges $\enorm{u-U_{\ell+1}}$ and $\enorm{U_{\ell+1}-U_\ell}$ with $\myd(u,U_{\ell+1})$ and $\myd(U_{\ell+1},U_\ell)$, respectively. Therefore, all the results from Section~\ref{section:optimality} hold (cf.\ the remarks after Theorem~\ref{thm:rconv} and the proof of Theorem~\ref{thm:optimal}) and, in particular, we obtain the optimality result of Theorem~\ref{thm:optimal}.

\bigskip
\textbf{Acknowledgement.}
The research of the authors is supported through the FWF project 
\emph{Adaptive Boundary Element Method}, funded by the Austrian Science fund (FWF) 
under grant P21732, see {\tt http://www.asc.tuwien.ac.at/abem}.

\newcommand{\bibentry}[2][]{\bibitem{#2}}

\end{document}